\pdfoutput=1 
\documentclass[a4paper,leqno,11pt]{amsart} 
\usepackage{amscd}
\usepackage{graphicx}
\usepackage{hyperref} 
\hypersetup{
    colorlinks=false,
    pdfborder={0 0 0},
}
\usepackage[scr=zapfc,scrscaled=1.2]{mathalfa}
\usepackage{mathrsfs}
\usepackage{tikz-cd}
\usepackage{enumitem}
\allowdisplaybreaks

\theoremstyle{plain}
\newtheorem{theorem}{Theorem}
\newtheorem{corollary}[theorem]{Corollary}
\newtheorem{lemma}[theorem]{Lemma}
\newtheorem{proposition}[theorem]{Proposition}

\numberwithin{theorem}{section}

\theoremstyle{definition}
\newtheorem{definition}[theorem]{Definition}
\newtheorem{example}[theorem]{Example}

\theoremstyle{remark}
\newtheorem{remark}[theorem]{Remark}
\newtheorem{notation}[theorem]{Notation}

\title{Elliptic normal curves of even degree and theta functions}

\date{March 31, 2023}

\author{Masanobu Kaneko} 
\address{Faculty of Mathematics, 
Kyushu University, 
Motooka 744, Nishi-ku,
Fukuoka 819-0395, 
Japan}
\email{mkaneko@math.kyushu-u.ac.jp}

\author{Masato Kuwata}
\address{Faculty of Economics, 
Chuo University, 
742-1 Higashinakano, 
Hachioji-shi, Tokyo 192-0393, 
Japan}
\email{kuwata@tamacc.chuo-u.ac.jp}

\newcommand\N{\mathbf {N}} \newcommand\Z{\mathbf {Z}}
\renewcommand\P{\mathbf {P}}  
\newcommand\R{\mathbf {R}} \newcommand\C{\mathbf {C}}
\renewcommand\H{\mathfrak{H}}
\renewcommand\Im{\operatorname{Im}}

\newcommand{\Gal}{\operatorname{Gal}}
\newcommand{\Proj}{\operatorname{Proj}}
\newcommand{\Aut}{\operatorname{Aut}}
\newcommand{\Hom}{\operatorname{Hom}}
\newcommand{\Div}{\operatorname{Div}}
\renewcommand{\div}{\operatorname{div}}
\newcommand{\addition}{\operatornamewithlimits{\mathrm{Sum}}}
\newcommand{\Nul}{\operatorname{Nul}}
\newcommand{\e}{\mathbf{e}}
\def\<{\langle}\renewcommand\>{\rangle}
\def\sltwo(#1,#2;#3,#4){\mathchoice%
{\begin{pmatrix} #1 & #2 \\ #3 & #4 \end{pmatrix}}
{\left(\begin{smallmatrix}#1 & #2 \\ #3 & #4 \end{smallmatrix}\right)}
{\left(\begin{smallmatrix}#1 & #2 \\ #3 & #4 \end{smallmatrix}\right)}
{\left(\begin{smallmatrix}#1 & #2 \\ #3 & #4 \end{smallmatrix}\right)}
}
\def\tsltwo(#1,#2;#3,#4){\mathchoice%
{\left(\begin{smallmatrix}#1 & #2 \\ #3 & #4 \end{smallmatrix}\right)}
{\left(\begin{smallmatrix}#1 & #2 \\ #3 & #4 \end{smallmatrix}\right)}
{\left(\begin{smallmatrix}#1 & #2 \\ #3 & #4 \end{smallmatrix}\right)}
{\left(\begin{smallmatrix}#1 & #2 \\ #3 & #4 \end{smallmatrix}\right)}
}

\makeatletter
\def\@fnsymbol#1{\ensuremath{\ifcase#1\or \dagger\or \ddagger\or
   \mathsection\or \mathparagraph\or \|\or \dagger\dagger
   \or \ddagger\ddagger \else\@ctrerr\fi}}\makeatother

\numberwithin{equation}{section}

\makeatother

\newcommand{\action}[1]{{}^{#1}\!}
\newcommand\ph{\varphi}
\newcommand{\dd}{D}
\renewcommand{\L}{\mathscr{L}}

\newcommand{\I}{\mathscr{I}}
\newcommand{\Sym}{\mathscr{S}}
\newcommand{\asterisk}{*}
\def\udots{\scalebox{-1}[1]{$\ddots$}}
\newcommand{\E}{E}

\begin{document}

\begin{abstract} 
An elliptic curve $E$ can be immersed in $\P^{N-1}$ as a curve of degree~$N$ by means of the linear system of $|NO|$, where $O$ is the origin of ~$E$. Well-known classical results going back to Bianchi and Klein say that if $N$ is odd, this immersion is uniquely determined by specifying a full-level $N$ structure. In this paper we show that if $N$ is even, uniqueness of immersion is ensured by specifying a level structure associated with a certain congruence subgroup between $\Gamma(N)$ and $\Gamma(2N)$. Moreover, we construct, over the complex number field, an immersion by means of suitably chosen theta functions, and write down the quadratic equations satisfied by them. 
\end{abstract}

\maketitle

\section{Introduction}

An elliptic normal curve of degree~$N$ means an elliptic curve $E$ together with an immersion into a projective space $E\hookrightarrow \P^{N-1}$ as a degree~$N$ curve that is contained in no hyperplane.  We assume $N\ge 4$, and that the base field $K$ is a field of characteristic not dividing~$N$. 
By the Riemann-Roch theorem, any elliptic curve $E$ over $K$ can be realized as an elliptic normal curve of degree~$N$ in $\P^{N-1}$ by means of a complete linear system $|D|$ with any effective divisor $D$ of degree~$N$.  In particular, we may take $D=NO$, where $O$ is the origin of the group structure of~$E$.  Moreover, it is known that the image of $E$ in $\P^{N-1}$ is defined by a system of $N(N-3)/2$ quadratic equations.

Classically, for the case $N=5$ over the complex number field $\C$, Bianchi~\cite{Bianchi} first wrote down defining equations of an elliptic curve in the following ``normal form'':
\[
E_{\phi} : 
\left\{\renewcommand{\arraystretch}{1.2}\begin{array}{l}
x_{0}^{2} + \phi x_{2}x_{3}-\phi^{-1}x_{1}x_{4} = 0, \\
x_{1}^{2} + \phi x_{3}x_{4}-\phi^{-1}x_{2}x_{0} = 0, \\
x_{2}^{2} + \phi x_{4}x_{0}-\phi^{-1}x_{3}x_{1} = 0, \\
x_{3}^{2} + \phi x_{0}x_{1}-\phi^{-1}x_{4}x_{2} = 0, \\
x_{4}^{2} + \phi x_{1}x_{2}-\phi^{-1}x_{0}x_{3} = 0.
\end{array}\right.
\]
These five quadratic equations define an elliptic curve $E_{\phi}$ as a projective curve of degree~$5$ in $\P^{4}$, and the translations by the $5$-torsion points are realized as linear transformations of the projective space.  The curve $E_{\phi}$ may be viewed as the universal elliptic curve over the modular curve parametrizing triples $(E,S,T)$ where $E$ is an elliptic curve and $(S,T)$ is a basis of $E[5]$ such that the Weil pairing $e_{5}(S,T)$ is equal to a fixed primitive $5$th root of unity $\zeta_{5}$.  The parameter $\phi=\phi(\tau)$ is a modular function  associated with the congruence subgroup $\Gamma(5)$.  Klein \cite{Klein} generalized Bianchi's results to obtain such a normal form of quadratic equations for general odd integers $N$, and studied their relation to modular functions of level~$N$. Later, V\'elu \cite{Velu:1978} studied the immersion of elliptic curves into $\P^{N-1}$ as elliptic normal curves of degree $N$ in terms of schemes. In the end, he constructed the modular curve $X(p)$ for odd prime number~$p$ as a subscheme in $\P^{p-1}$ over $\operatorname{Spec}\Z[1/p]$, and the universal elliptic curve $E(p)$ over $X(p)$. (See also \cite{Fisher:thesis}.) 

In the case where $N$ is even, Hurwitz \cite{Hurwitz} generalized Klein's results to a certain extent.  In this case, however, it is not possible to obtain the universal family associated with $\Gamma(N)$.  Hurwitz's family is over the modular curve associated with some congruence subgroup of level $2N$ or $4N$ depending on $N\bmod 4$.
In this paper we use the work of V\'elu \cite{Velu:1978}, and construct a slightly different immersion from that of Hurwitz based on an analysis of the level structures associated with congruence subgroups between $\Gamma(N)$ and $\Gamma(2N)$.  We obtain a canonical coordinate system compatible with the translations by the $N$-torsion points and the universal elliptic curve associated with a certain congruence subgroup.  

Over $\C$, we go one step further and realize the above immersion $E\to \P^{N-1}$ in terms of certain theta functions that become the coordinate functions of the system described above, and write down the quadratic equations satisfied by the image of $E$ using the classical relations among theta functions.

Let us describe our results in more detail.  Let $E$ be an elliptic curve over a field of characteristic not dividing~$N$, and consider the immersion $E\to \P^{N-1}$ using the complete linear system $|NO|$.  For an $N$-torsion point $T\in E[N]$, let us denote by $\tau_{T}$ the translation-by-$T$ map $P\mapsto P+T$.  It is easy to see that $\tau_{T}$ can be lifted to an automorphism of the ambient space $\P^{N-1}$.  By choosing a basis $(S,T)$ of $E[N]$, we can find a coordinate system such that the translations $\tau_{S}$ and $\tau_{T}$ are realized as linear transformations of the ambient space $\P^{N-1}$.  

If $N$ is odd, the choice of such a coordinate system is unique once we fix a $\Gamma(N)$ structure, that is, a basis $(S,T)$ of $E[N]$ such that $e_{N}(S,T)$ is equal to a fixed primitive $N$th root of unity $\zeta_{N}$ (Proposition~\ref{prop:proj-coord-odd}).  Thus, we obtain an immersion of the family of elliptic curves with $\Gamma(N)$ structure to a single $\P^{N-1}$, and by associating the origin of the elliptic curves, we obtain a morphism from the modular curve $X(N)$ to $\P^{N-1}$. 
  
The situation is not the same if $N$ is even.  In this case, the coordinate system mentioned above is not unique even if we specify a $\Gamma(N)$ structure. So, we define an intermediate subgroup $\Gamma^{(N)}(2N)$ between $\Gamma(N)$ and $\Gamma(2N)$ (see Definition~\ref{def:Gamma^(N)(2N)}), and show that the coordinate system mentioned above is determined uniquely once we fix a $\Gamma^{(N)}(2N)$ structure (Theorem~\ref{thm:proj-coord-even}).  As a consequence we obtain the universal elliptic curve over the modular curve associated with $\Gamma^{(N)}(2N)$.

Furthermore, over $\C$, we realize the immersion $E\hookrightarrow \P^{N-1}$ using the theta functions, denoted by $\theta_{k}^{(N)}(z,\tau)$ for $k=0,\dots,N-1$ (see Definition~\ref{def:theta-N}), that serve as the coordinate functions of the coordinate system described above, and we obtain quadratic equations of the image (Theorems~\ref{thm:eq-N-odd} and~\ref{thm:eq-N-even}) using the relations between the theta functions coming from Jacobi's identity \eqref{eq:Jacobi4} or \eqref{eq:Jacobi2}, which is essentially the addition formula for the elliptic curve.

In \S\S\ref{sec:level-4}--\ref{sec:level-8} we work out in detail the cases $N=4,6$, and $8$.  There, we show the explicit equations of the modular curves and the universal curves over~them.

Since our theta functions and Hurwitz's $\sigma$-functions \cite{Hurwitz} look quite different, it is not easy to see the exact relationship between them.  In Appendix~\ref{appendix}, we describe the connections and differences between these two sets of functions in detail.

\subsection*{Acknowledgements}
Kaneko was supported in part by JSPS KAKENHI Grant Numbers JP23340010, JP15K13428, JP16H06336. 
Kuwata was supported by JSPS KAKENHI Grant Numbers JP23540028, JP26400023, and by the Chuo University Grant for Special Research.  Part of this work was done while Kuwata was visiting Boston University.  
We also thank the referee for useful comments.

\section{Preliminaries}

From this section until \S\ref{sec:can-coord-system} we fix a positive integer $N$ greater than~$3$, and suppose that the base field $K$ is a field of characteristic not dividing $N$.  Let $K_{s}$ be a separable closure of $K$, and let $\Gamma_{K}=\Gal(K_{s}/K)$.  

Let $E$ be an elliptic curve defined over $K$. Denote by $E(K)$ the group of $K$-rational points of~$E$, and $E[N]=\{P\in E(K_{s})\mid NP=O\}$ the subgroup of $N$-torsion points. 
Since we use the notation and results of V\'elu \cite{Velu:1978} extensively, we give a summary of parts of \cite{Velu:1978} necessary for later use.

\subsection{The Weil pairing}\label{subsec:Weil_pairing}

We first fix some notation.
\begin{itemize}
\item $K(E)^{\times}$ : the multiplicative group of the function field $K(E)$ of $E$.
\item $\Div_{K}(E)$ : the group of $K$-rational divisors of $E$; i.e., the group of formal $\Z$-linear combination%
\footnote{For a point $P\in E(K_{s})$, we denote by $\{P\}$ the base of the formal sums associated with~$P$.  Thus, the sum $\{P\}+\{Q\}$ is a formal sum, while the sum in $\{P+Q\}$ means the addition in the elliptic curve~$E$.} 
 $D=\sum_{P\in E(K_{s})} n_{P}\{P\}$ with $n_{P}\in \Z$, such that $n_{P}=0$ for all but finitely many $P\in E(K_{s})$, and $n_{\, \action{\sigma}P}=n_{P}$ for all $\sigma\in \Gamma_{K}=\Gal(K_{s}/K)$.  
\item $\Div_{K}^{0}(E)$ : the kernel of $\deg:\Div_{K}(E)\to \Z$. 
\item $\Nul_{K}(E)$ : the kernel of the homomorphism $\addition:\Div_{K}(E)\to E(K)$ defined by $\sum_{P}n_{P}\{P\}\mapsto\addition\limits_{P}\,n_{P}P$, where $\addition$ means the addition in~$E$.
\item $P_{K}(E) = \Div_{K}^{0}(E)\cap \Nul_{K}(E)$.
\end{itemize}

We have an obvious homomorphism $\div:K(E)^{\times}\to \Div_K(E)$ that sends a function to its divisor, and the following is fundamental.

\begin{theorem}[Abel-Jacobi]\label{th:abel}
The following sequence is exact.
\begin{equation}\label{eq:Abel}
1 \longrightarrow K^{\times} \longrightarrow K(E)^{\times}
\overset{\div}{\longrightarrow} P_{K}(E) \longrightarrow 0.
\qedhere
\end{equation}
\end{theorem}

In this section and the next we assume $E[N]\subset E(K)$, so in particular, $K$ contains the $N$th roots of unity.
Let $G\subset E[N]$ be a cyclic group of order~$N$.  We consider \eqref{eq:Abel} as an exact sequence of $G$-modules. 
Define an action of $G$ on various groups as follows:
\begin{itemize}
\item trivially on $K^{\times}$, $\Z$, and $E(K)$.
\item by translation on $K(E)$; i.e., if $f\in K(E)$ and $T\in G$, define $\action{T}\!f(X)=f(X-T)$, where $X$ is a generic point of~$E$.
\item by translation on $\Div_{K}(E)$; i.e., if $\dd=\sum_{P}n_{P}\{P\}$ and $T\in G$, define 
\[
\action{T}\!\dd=\sum_{P}n_{P}\{P+T\}.  
\]
\end{itemize}
Clearly, the action of $G$ on $\Div_{K}(E)$ induces actions on $\Div_{K}^{0}(E)$ and $P_{K}(E)$.
The Weil pairing is obtained essentially as the connecting homomorphism $\delta$ of the long exact sequence of the group cohomology induced from \eqref{eq:Abel}:
\begin{equation}\label{eq:long-exact-seq}
K^{\times} \longrightarrow {K(E)^{\times}}^{G} \overset{\div}\longrightarrow P_{K}(E)^{G}
\overset{\delta}{\longrightarrow} H^{1}(G,K^{\times}).
\end{equation}

\begin{theorem}[Weil pairing]\label{thm:weil-pairing}
\begin{enumerate}
\item
The homomorphism $\delta$ in \eqref{eq:long-exact-seq} induces  a $\Gamma_{K}$-isomor\-phism $\Psi:G'=E[N]/G \overset{\simeq}{\longrightarrow}\Hom(G,\mu_{N})$.
\item
For $S'\in G'$ and $T\in G$, define $e_{G}(S',T)$ by
\[
e_{G}(S',T)=\Psi(S')(T).
\]
Then $e_{G}$ is a $\Gamma_{K}$ compatible non degenerate bilinear form $G'\times G\to \mu_{N}$.
\item
For $S'\in G'$, choose $S\in E[N]$ such that $S\bmod G=S'$, and define $\dd_{S}=\{S\}-\{O\}\in \Div_{K_{s}}^{0}(E)$.  Choose $f_{S}\in K_{s}(E)^{\times}$ such that 
\[
\div f_{S} = \sum_{T\in G}\action{T}\!\dd_{S} 
= \sum_{T\in G}\bigl(\{S+T\}-\{T\}\big).
\]
Then, we have%
\[
e_{G}(S',T)=\action{T}\!f_{S}/f_{S}=f_{S}(X-T)/f_{S}(X).
\]
\end{enumerate}
\end{theorem}

\begin{proof}
Since $G$ is a subgroup of $E[N](K)$ by assumption, it follows that $\mu_{N}\in K^{\times}$, and thus we have $H^{1}(G,K^{\times})\simeq \Hom(G,\mu_{N})$.  The proof is a standard diagram chase.  See V\'elu \cite[Ch.~1]{Velu:1978} for detail.
\end{proof}

The usual Weil pairing $e_{N}:E[N]\times E[N]\to \mu_{N}$ is defined in a similar manner.  The relation between $e_{G}$ and $e_{N}$ is given by
\[
e_{N}(S,T)=e_{G}(\phi(S),T) \quad \text{for $S\in E[N]$, $T\in G$},
\]
where $\phi$ is the isogeny $E\to E/G$.

\subsection{Central extension $E[N](\dd)$}\label{central_ext}

A central extension of a group $G$ is a short exact sequence of groups
$1\to A\to H\to G\to 1$ such that $A$ is in the center of the group $H$.
Here, we consider central extensions of $E[N]$ by $K^{\times}$, that is, short exact sequences of groups
\[
1 \longrightarrow K^{\times} \longrightarrow H 
\longrightarrow E[N] \longrightarrow 0.
\]
(Note that the operation of $K^{\times}$ is written multiplicatively, while that of $E[N]$ is written additively.)

Given a divisor $\dd\in\Div_{K}(E)$, we may construct such a central extension.  

\begin{definition}\label{def:extension}
Let $\dd\in \Div_{K}(E)$ be a divisor of degree divisible by~$N$.  Define 
\[
E[N](\dd)=\{(T,f)\in E[N]\times K(E)^{\times}\mid \div f =\action{T}\!\dd-\dd\}, 
\]
with a group operation on $E[N](\dd)$ given by
\[
(S,g)(T,f)=(S+T,g\cdot\action{S}\!f).
\]
\end{definition}

Indeed, we have a natural inclusion $K^{\times}\to E[N](\dd)$ given by $c\mapsto (O,c)$, and an exact sequence
\[
1 \longrightarrow K^{\times} \longrightarrow E[N](\dd) 
\longrightarrow E[N] \longrightarrow 0.
\]
From now on, we identify $(O,c)\in E[N](\dd)$ with $c\in K^{\times}$ and consider $K^{\times}$ as a subgroup in $E[N](\dd)$ (to ease the notation).


\begin{lemma}\label{lem:lin-equiv}
If two divisors $\dd_{1}$ and $\dd_{2}$ are linearly equivalent, then the extensions $E[N](\dd_{1})$ and $E[N](\dd_{2})$ are isomorphic.
\end{lemma}

\begin{proof} Straightforward.
\end{proof}

In general $E[N](\dd)$ is not commutative and contains elements of infinite order.  
For $(S,g),(T,f)\in E[N](D)$, we denote the commutator and the $N$th power by 
\begin{align}\label{eq:<T,T'>}
&\<S,T\> =(S,g)(T,f)(S,g)^{-1}(T,f)^{-1}
=(g/\action{T}\!g)\cdot(\action{S}\!f/f)
\\
&v(T)=(T,f)^{N}=f\cdot\action{T}f\cdot\action{2T}f\cdot\dots\cdot\action{(N-1)T}f.
\end{align}
Recall that we identify $(O,c)$ with $c$.  As this notation suggests,  the facts that both $\<S,T\>$ and $v(T)$  $\bmod$ $n$th powers are independent of the choice of $f$ and $g$ are shown by the following proposition and lemma of V\'elu. 

\begin{proposition}[V\'elu \hbox{\cite[Prop.~2.3]{Velu:1978}}]\label{prop:<S,T>}
Let $\dd$ be a divisor of degree divisible by $N$, and let $S,T\in E[N]$.  Then, $\<S,T\>$ is a bilinear form on $E[N]$ with its value in $\mu_{N}$, and 
\[
\<S,T\> = e_{N}(T,S)^{\frac{\deg\dd}{N}}.
\qedhere
\]
\end{proposition}

\begin{lemma}[V\'elu \hbox{\cite[Lemma~2.4]{Velu:1978}}]\label{lem:v(T)}
Let $\dd$ be a divisor of degree divisible by $N$, and let $T\in E[N]$.  Then, $v(T)$ is in $K^{\times}$, and $v(T)\bmod {K^{\times}}^{N}$ is independent of the choice of $f$.
\qedhere
\end{lemma}

The case where $\dd=N\{O\}$ is of particular interest.   In this case, the results depend on the parity of~$N$.  The following lemma by V\'elu will be needed in \S\ref{subsec:even-case}.

\begin{lemma}
\label{lem:E[N](NO)}
Suppose $(T,f)\in E[N](N\{O\})$, and $\tilde T\in E[2N]$ is a point such that $2\tilde T=T$.  Then, 
\begin{enumerate}
\item {\upshape (V\'elu \cite[Lemma~2.8]{Velu:1978})} \ $v(T) = (T,f)^{N} = f(\tilde T)^{N}$.
\item {\upshape (V\'elu \cite[Lemma~2.7]{Velu:1978})} \ If $U_{2}\in E[2]-\{O\}$, then 
\[
\action{U_{2}}f(\tilde T)=(-1)^{N}e_{2}(U_{2},N\tilde T)f(\tilde T).
\]
\end{enumerate}
\end{lemma}

\section{Projective immersion associated with a cyclic subgroup}\label{sec:proj-immersion}

%
Let $\dd\in \Div_{K}(E)$ be a divisor.  If $\deg \dd\ge 3$, it is well known that the complete linear system $|\dd|$ gives an immersion $E\hookrightarrow |\dd|$ in a projective space.  Here, we take a closer look at this fact.  The following is a summary of results of V\'elu \cite{Velu:1978}.

\begin{definition}\label{def:L(D)}
For $\dd\in \Div_{K}(E)$ and nonnegative integer $d$, define
\begin{align*}
&\L^{d}(\dd) =\{ h\in K(E)^{\times}\mid \div h + d\dd\geq 0\} \cup \{0\}
\\
&\L^{\asterisk}(\dd)=\textstyle{\bigoplus_{d\ge0}}\,\L^{d}(\dd),
\\
&\Sym^{\asterisk}(\dd)=\operatorname{Sym}^{\asterisk}\L^{1}(\dd).
\end{align*}
\end{definition}
Here, $\operatorname{Sym}^{\asterisk}V$ means the symmetric algebra on a vector space $V$.
By the Riemann-Roch theorem, $\L^{d}(\dd)$ is a $K$-vector space of dimension $d\deg D$, and the space $\L^{\asterisk}(\dd)$ is equipped naturally with a structure of graded algebra by the multiplication of functions.
There is a canonical homomorphism of graded algebras $\Sym^{\asterisk}(\dd)\to \L^{\asterisk}(\dd)$, and let $\I^{\asterisk}(\dd)$ be its kernel. We have the exact sequence
\begin{equation}\label{eq:short-ex}
0\longrightarrow \I^{\asterisk}(\dd) \longrightarrow \Sym^{\asterisk}(\dd)
\longrightarrow \L^{\asterisk}(\dd).
\end{equation}

\begin{proposition}\label{prop:proj-normality}
Let $D$ be an effective divisor of degree $N\geq 4$.
\begin{enumerate}
\item {\upshape (V\'elu \cite[Th.~3.3]{Velu:1978}) }
The exact sequence \eqref{eq:short-ex} extends to the following exact sequence:
\[
0\longrightarrow \I^{\asterisk}(\dd) \longrightarrow \Sym^{\asterisk}(\dd) 
\longrightarrow \L^{\asterisk}(\dd) \longrightarrow 0.
\]
\item
$\dim\I^{2}(\dd)=N(N-3)/2$.
\item {\upshape (V\'elu \cite[Th.~3.9]{Velu:1978}) }
$\I^{\asterisk}(\dd)$ is generated by $\I^{2}(\dd)$, i.e., $\I^{\asterisk}(\dd)=\I^{2}(\dd)\cdot \Sym^{\asterisk}(\dd)$.
\item
The image of the map 
$E\hookrightarrow \Proj \Sym^{\asterisk}(\dd)\simeq \P^{N-1}$ is the scheme-theoretic intersection of the quadrics that contain the image of~$E$.
\end{enumerate}
\end{proposition}

\begin{proof}
(1) \ The exactness at $\L^{\asterisk}(\dd)$ is one of the equivalent definitions of projective normality.  For a proof see V\'elu \cite[Th.~3.3]{Velu:1978}.  See also Hartshorne \cite[Ex.~IV.4.2]{Hartshorne:AG} and Mumford \cite[p.~55]{Mumford:1970a}.  The formula (2) follows immediately from (1) as $\dim \Sym^{2}(\dd)=N(N+1)/2$ and $\dim \L^{2}(\dd)=2N$.

(3) \ See V\'elu \cite[Th.~3.9]{Velu:1978}.  The assertion (4) is just a paraphrase of    (3) as the image of $E$ is defined by an ideal generated by $\I^{2}(\dd)$.
\end{proof}

The central extension $E[N](\dd)$ has a representation on $\L^{\asterisk}(\dd)$.  More precisely, V\'elu showed the following.
\begin{proposition}[V\'elu {\cite[Prop.~2.13]{Velu:1978}}]
\label{prop:LD-rep}
\begin{enumerate}
\item
For $h\in \L^{d}(\dd)$ and $(T,f)\in E[N](\dd)$, the function $\action{T}h\cdot f^{d}$ is in $\L^{d}(\dd)$.
\item
Let $\tau_{d}(T,f)$ be the automorphism defined by $h\mapsto f^{d}\cdot\action{T}h $ of $\L^{d}(\dd)$.  Then, $\tau_{d}:(T,f)\mapsto \tau_{d}(T,f)$ is a representation of $E[N](\dd)$ on $\L^{d}(\dd)$, and the representations $\tau_{d}$ extend to a representation $\tau$ of $E[N](\dd)$ on the graded ring $\L^{\asterisk}(\dd)$. 
\item
If $\dd$ is an effective divisor of degree $\geq 2$ and $d\geq 1$, the kernel of the representation $\tau_{d}$ consists of the elements $(O,f)$, where $f$ is a constant satisfying $f^{d}=1$. In particular, $\tau_{1}$ is faithful.
\end{enumerate}
\end{proposition}

Let $C\subset E[N]$ be a cyclic subgroup of order $N$.  Suppose $\dd$ is an effective divisor of degree $N$ invariant under the translations by $C$.  Then, we may choose a particular coordinate system of $\P^{N-1}$ so that the immersion $E\hookrightarrow |\dd|\simeq \P^{N-1}$ is expressed in a simple way.  To do so, we consider the decomposition of $\L^{1}(\dd)$ into eigenspaces.

\begin{definition}\label{def:L(D,chi)}
For any character $\chi\in \Hom(C,\mu_{N})$ and a positive integer~$d$, define
\[
\L^{d}(\dd,\chi) = \{h\in \L^{d}(\dd)\mid
\action{U}h=\chi(U)\cdot h \text{ for all $U$ in $C$} \}.
\]
Since $\L^{d}(\dd)$ is a finite dimensional vector space on which the cyclic group $C$ acts, it decomposes into eigenspaces.  Thus, we have
\[
\L^{d}(\dd) = \textstyle{\bigoplus_{\chi}}\L^{d}(\dd,\chi),
\quad
\text{and}
\quad
\L^{\asterisk}(\dd)=\textstyle{\bigoplus_{d,\chi}}\L^{d}(\dd,\chi).
\]
\end{definition}

\begin{definition}\label{def:char_S}
Let $C\subset E[N]$ be a cyclic subgroup of order $N$.  For any $S\in E[N]$, define the character $\chi_{S}\in \Hom(C,\mu_{N})$ by
\[
\chi_{S}(U)=e_{N}(S,U) \quad\text{for all $U\in C$}.
\]
\end{definition}

\begin{proposition}\label{prop:L(D)}
Let $C\subset E[N]$ be a cyclic subgroup of order $N$, and $\dd$ an effective divisor of degree $N$ invariant under the translations by $C$.
\begin{enumerate}
\item
The space $\L^{1}(\dd,\chi)$ is $1$-dimensional.
\item 
If $(S,f_{S})\in E[N](\dd)$, then $f_{S}$ is a basis of $\L^{1}(\dd,\chi_{S})$.
\item
The map $\tau_{d}(S,f_{S}):\L^{d}(\dd,\chi)\to \L^{d}(\dd,\chi\chi_{S}^{d})$ is an isomorphism for all $d\ge 0$.
\end{enumerate}
\end{proposition}

\begin{proof}
Let $(S,f_{S})$ be an element of $E[N](\dd)$. Since $\div f_{S} + \dd=\action{S}\dd\ge0$, $f_{S}$ is in $\L^{1}(\dd)$.
Since $\action{U}D-D=0$ for $U\in C$, we have $(U,1)\in E[N](\dd)$.  By \eqref{eq:<T,T'>}, we have $\<U,S\>={\action{U}\!f_{S}}/{f_{S}}$, and by Proposition~\ref{prop:<S,T>}, we have $\<U,S\>=e_{N}(S,U)=\chi_{S}(U)$.  Thus, we have $\action{U}\!f_{S}=\chi_{S}(U)f_{S}$, which shows $f_{S}\in \L^{1}(\dd,\chi_{S})$.

If $h\in \L^{d}(\dd,\chi)$, then 
\[
\action{U}(\action{S}h\cdot f_{S}^{d}) = \action{U+S}h\cdot(\action{U}\!f_{S})^{d}=\chi(U)\cdot\action{S}h\cdot(\action{U}\!f_{S})^{d}
=\chi(U)\chi_{S}(U)^{d}(\action{S}h\cdot f_{S}^{d}).
\]
This implies the automorphism $\tau_{d}(S,f_{S})$ maps $\L^{d}(\dd,\chi)$ onto $\L^{d}(\dd,\chi\chi_{S}^{d})$, and (3) is proved. As a consequence, $\tau_{d}$ permutes the spaces $\L^{d}(\dd,\chi)$ transitively as long as $d$ is relatively prime to~$N$, and in that case the spaces $\L^{d}(\dd,\chi)$ have the same dimension independent of~$\chi$.  In particular, this is the case for $\L^{1}(D,\chi)$.

By the Riemann-Roch theorem, $\dim \L^{1}(\dd)=\deg\dd=N$.  On the other hand
\begin{align*}
\dim \L^{1}(\dd) &=
\sum_{\chi}\dim \L^{1}(\dd,\chi) 
= \#\Hom(C,\mu_{N}) \cdot\dim \L^{1}(\dd,\chi_{0}) 
\\
&= N \dim \L^{1}(\dd,\chi_{0}),
\end{align*}
where $\chi_{0}$ is the trivial character.  Thus, we conclude that $\dim \L^{1}(\dd,\chi)=1$ for any $\chi$, which proves (1).  Since any nonzero function in a one dimensional space is its basis, $f_{S}$ is a basis of $\L^{1}(\dd,\chi_{S})$.
\end{proof}

\section{Canonical coordinate system}\label{sec:can-coord-system}

In this section we consider the projective immersion $E\hookrightarrow \bigl|N\{O\}\bigr|\simeq \P^{N-1}$.  We show that, by fixing a certain level structure on $E$, we can choose a unique coordinate system with prescribed properties.  The situation differs depending on the parity of~$N$.  The odd case is well known, but we include it here for the comparison with the even case.

\subsection{Odd case}\label{subsec:odd-case}

Throughout this paragraph we assume $N$ is an odd integer $\ge 3$.  This section serves as a prototype for the even case, and all the material in this section is written in V\'elu \cite{Velu:1978}.  See also Fisher \cite{Fisher:thesis}, \cite{Fisher:2001} for geometric treatment.

Let $(S,T)$ be a pair of points in $E[N]$, and $\sltwo(a,b;c,d)$ in $M_{2}(\Z)$.  Consider the right action of $\sltwo(a,b;c,d)$ on $(S,T)$ defined by
\[
(S,T)\sltwo(a,b;c,d)=(aS+cT,bS+dT).
\]
Recall that if the pair $(S, T)$ is a basis of $E[N]$ satisfying $e_{N}(S,T)=\zeta$, then the pair $(S',T')=(aS+cT,bS+dT)$ is once again a basis of $E[N]$ satisfying $e_{N}(S',T')=\zeta$ if and only if $\sltwo(a,b;c,d)\in SL_{2}(\Z)$.  The kernel of this action is given by
\[
\Gamma(N) = \left\{\left.
\sltwo(a,b;c,d)\in SL_{2}(\Z) \,\right|\,
\sltwo(a,b;c,d)\equiv \sltwo(1,0;0,1)\bmod N\right\},
\]
which is the principal congruence subgroup of level $N$ in $SL_{2}(\Z)$. 

\begin{definition}[$\Gamma(N)$ structure]\label{def:level-Gamma(N)}
Fix a primitive $N$th root of unity $\zeta$. A \emph{$\Gamma(N)$-structure} on an elliptic curve $E$ is a pair of $N$-torsion points $(S, T)$ satisfying $e_{N}(S,T)=\zeta$.
\end{definition}

From now on we fix a primitive $N$th root of unity $\zeta$ in $K_{s}$, and a $\Gamma(N)$-structure $(S,T)$.  Define the divisor $\dd_{T}$ by
\[
\dd_{T}=\sum_{n=0}^{N-1}\{nT\}.
\]
Clearly, $\dd_{T}$ is invariant under the translation by~the cyclic group $C=\<T\>$.

Since $\addition_{n=0}^{N-1}nT=O$, $\dd_{T}$ is linearly equivalent to~$N\{O\}$.  Choose $f_{S,T}$ such that
\(
\div f_{S,T}=\action{S}\dd_{T}-\dd_{T}.
\)
Since $(S,f_{S,T})^{N}\equiv 1 \bmod {K^{\times}}^{N}$\! by Lemma~\ref{lem:E[N](NO)} (1), by multiplying $f_{S,T}$ by a suitable constant if necessary, we may assume $(S,f_{S,T})^{N}=1$.  

\begin{definition}\label{def:X_k-odd}
Let $f_{S,T}$ be as above. 
Define functions $X_{0},X_{1},\dots,$ $X_{N-1}\in K(E)$ indexed by the elements $k\in\Z/N\Z$ by 
\[
X_{0} = 1,
 \quad
X_{k}
=\tau_{1}(S,f_{S,T})X_{k-1} = \action{S}X_{k-1}\cdot f_{S,T}
\quad  (k=1,\dots,N-1),
\]
where $\tau_{1}(S,f_{S,T})$ is the automorphism of $\L(D)$ defined in Proposition~\ref{prop:LD-rep}{\color{red}.}
\end{definition}

\begin{lemma}\label{lem:X_k-odd}
\begin{enumerate}
\item $X_{k}$ is well-defined, i.e., $\tau_{1}(S,f_{S,T})X_{N-1}=X_{0}$.
\item $\div X_{k}=\action{kS}D_{T}-D_{T}$ for $k\in \Z/N\Z$.
\item 
The action of $(T,1)\in E[N](\dd_{T})$ is given by
\[
\tau_{1}(T,1)X_{k}=\action{T}X_{k}=e_{N}(S,T)^{k}X_{k}.
\]
\item
$X_{k}$ is a basis of $\L^{1}(\dd_{T},\chi_{S}^{k})$ for $k\in \Z/N\Z$. 
\item
$X_{k}(-P)=X_{-k}(P)$ for any $P\in E$.
\end{enumerate}
\end{lemma}

\begin{proof}
The assertion (1) follows from the fact that $(S,f_{S,T})^{N}=1$.  By definition we have $X_{k}=\action{S}X_{k-1}\cdot f_{S,T}$, and then the assertion (2) follows by induction on~$k$.  By Proposition~\ref{prop:<S,T>}, we have $\<T,S\>=e_{N}(S,T)=\chi_{S}(T)$, and thus 
\[
\tau_{1}(T,1)X_{1}
=\tau_{1}(T,1)\tau_{1}(S,f_{S,T})X_{0}
=\chi_{S}(T)
\tau_{1}(S,f_{S,T})\tau_{1}(T,1)X_{0}
=\chi_{S}(T)X_{1}.
\]
The assertion (3) then follows by induction on~$k$, and (4) follows immediately from (3) and Proposition~\ref{prop:L(D)}\,(2). 
To prove (5), it suffices to remark $[-1](\action{kS}\dd_{T})=\action{-kS}\dd_{T}$.
This implies that the functions $X_{k}\circ [-1]$ and $X_{-k}$ has the same divisor and thus $X_{-k}=c_{k}X_{k}\circ [-1]$ for some constants $c_{k}$. It is clear that $c_{0}=1$. Since we have $\tau_{1}(S,f_{S,T})\circ[-1]\circ \tau_{1}(S,f_{S,T})=[-1]$, we can prove $c_{k}=1$ by induction on~$k$.
\end{proof}

Interpreting the above proposition geometrically, we have 

\begin{proposition}\label{prop:proj-coord-odd}
Let $E$ be an elliptic curve, and $E\hookrightarrow\P^{N-1}$ be an immersion as an elliptic normal curve of odd degree~$N$ via the complete linear system $|N\{O\}|$.  Choose a primitive $N$th root of unity $\zeta$ and a $\Gamma(N)$-structure $(S,T)$. Then, there exists a unique coordinate system of $\P^{N-1}$ such that the translation maps $\tau_{S}$, $\tau_{T}$, and the multiplication-by-$(-1)$ map $[-1]$ are given by the following elements $M_{S}$, $M_{T}$ and $M_{[-1]}$ in $PGL_{N}(K(\zeta))$, respectively.
\begin{equation}\label{eq:matrices}
\setlength{\arraycolsep}{3pt}
\begin{gathered}
M_{S}= 
\renewcommand{\arraystretch}{0.75}
\left[\begin{array}{*5c|c}
0 &  & \cdots &  & 0 &  1\\\hline
1 &  &  &  &  &  0 \\
  & 1 &  &   &   &  \\
 &  & \ddots &   &   &\vdots \\
 &  &  &  1 &   &   \\
 &  &  &  & 1 & 0
\end{array}\right],
\quad
M_{T}=
\renewcommand{\arraystretch}{0.9}
\left[\begin{array}{*5c}
1 &   &   &   &   \\
  & \zeta &   &   &   \\
  &   & \zeta^{2} &   &   \\
  &   &   & \ddots &   \\ 
  &   &   &   & \zeta^{N-1}
\end{array}\right],
\\
\renewcommand{\arraystretch}{0.75}
M_{[-1]}=\left[\begin{array}{c|*5c}
1 & 0 &   & \cdots  &   & 0 \\\hline
0 &   &   &   &   & 1 \\
  &   &   &   & 1 &  \\
\vdots  &   &   & \udots &   &   \\
  &   & 1 &   &    &   \\
0 & 1 &   &   &   &  
\end{array}\right]
\end{gathered}
\end{equation}
\end{proposition}

\begin{remark}
When we use the square bracket $[\quad ]$ for matrix, we mean it is a class in $PGL$ or $PSL$.
\end{remark}

\begin{proposition}\label{prop:quad-eqs-odd}
The notation being as above, we have the following.
\begin{align*}
&\I^{2}(\dd_{T}) 
= \textstyle{\bigoplus_{k\in\Z/N\Z}}\I^{2}(\dd_{T},\chi_{S}^{k}),\\
&\I^{2}(\dd_{T},\chi_{S}^{k})=\I^{2}(\dd_{T})\cap \<X_{i}X_{j}\mid i+j\equiv k \mod N\>,
\\
&\dim \I^{2}(\dd_{T},\chi_{S}^{k}) =(N-3)/2, \quad
\dim \I^{2}(\dd_{T}) =N(N-3)/2.
\end{align*}
\end{proposition}

\begin{proof}
These assertions are the contents of \S3.2 of \cite{Velu:1978} expressed in our notation.  Since $N$ is odd, there is only one $\chi'$ such that $\chi'^{2}=\chi$ for any character $\chi\in \Hom(C,\mu_{n})$.  Thus we have $\dim \I^{2}(\dd_{T},\chi_{S}^{k}) =(N-3)/2$ by Corollaire~3.6 of~\cite{Velu:1978}.
\end{proof}

\subsection{Even case}\label{subsec:even-case}

We now assume that $N$ is a positive even integer.
A critical difference between this case and the last is that we have $\addition_{n=0}^{N-1}nT=T_{2}\neq O$ if $N$ is even, and thus $\dd_{T}=\sum_{n=0}^{N-1}\{nT\}$ is \emph{not} linearly equivalent to~$N\{O\}$.  To define a divisor linearly equivalent to $N\{O\}$, we use a $\tilde T\in E[2N]$ such that $2\tilde T=T$.

First we fix a primitive $N$th root of unity $\zeta$ and a $\Gamma(N)$-structure $(S,T)$, then we consider $\Gamma(2N)$-structures on $E$ in relation to this fixed $\Gamma(N)$-structure.  To do so, we first choose a primitive $2N$th root of unity $\tilde\zeta$ such that $\tilde\zeta^{2}=\zeta$.

\begin{definition}\label{def:level-2N}
Suppose $(S,T)$ is a $\Gamma(N)$-structure on $E$ with $e_{N}(S,T)=\zeta=\tilde \zeta^{2}$. 
\begin{enumerate}
\item
We say that a $\Gamma(2N)$-structure $(\tilde S,\tilde T)$ on $E$ is \emph{above} $(S,T)$ if 
\[
2\tilde S=S, \quad 2\tilde T=T, \quad\text{and}\quad
e_{2N}(\tilde S,\tilde T)=\tilde\zeta.
\]
\item
We say that two $\Gamma(2N)$-structures $(\tilde S,\tilde T)$ and $(\tilde S',\tilde T')$ are \emph{similar} if $(\tilde S',\tilde T')$ equals either
\[
(\tilde S,\tilde T) \quad\text{or}\quad
(\tilde S,\tilde T)\sltwo(1+N,0;0,1+N).
\]
We write $(\tilde S',\tilde T')\sim (\tilde S,\tilde T)$.
\end{enumerate}
\end{definition}

\begin{notation}\label{notation:2-torsion}
For simplicty, we denote by $U_{2}$ the $2$-torsion point $\frac{N}{2}U\in E[2]$ for any $U\in E[N]$.  
\end{notation}

Note that if $(S,T)$ is a $\Gamma(N)$-structure, then $(S_{2},T_{2})$ is a basis of $E[2]$, and $e_{2}(S_{2},T_{2})=e_{N}(S,T)^{N/2}=\zeta^{N/2}=-1$.

\begin{lemma}\label{lem:level-N-2N}
For a given $\Gamma(N)$-structure $(S,T)$, there are eight $\Gamma(2N)$-structures $(\tilde S,\tilde T)$ above $(S,T)$.  Up to similarity, these eight are classified into four classes.  If $(\tilde S,\tilde T)$ is one of them, then the following four represent the four different classes\/:
\[
(\tilde S,\tilde T), \quad
(\tilde S,\tilde T)\sltwo(1,N;0,1),\quad
(\tilde S,\tilde T)\sltwo(1,0;N,1),\quad
(\tilde S,\tilde T)\sltwo(1,N;N,1),
\]
or, using Notation~\ref{notation:2-torsion},
\[
(\tilde S,\tilde T), \quad
(\tilde S,\tilde T + S_{2}),\quad
(\tilde S + T_{2},\tilde T),\quad
(\tilde S + T_{2},\tilde T + S_{2}).
\]
\end{lemma}

\begin{proof}
There are sixteen pairs $(\tilde S,\tilde T)$ satisfying $2\tilde S=S$ and $2\tilde T=T$. Choose one pair $(\tilde S,\tilde T)$.  Replacing $\tilde S$ by $\tilde S+S_{2}$ if necessary, we may assume $e_{2N}(\tilde S,\tilde T)=\tilde\zeta$.  Thus, we obtain at least one $\Gamma(2N)$-structure above $(S,T)$.  All the sixteen pairs $(\tilde S',\tilde T')$ satisfying $2\tilde S'=S$ and $2\tilde T'=T$ are
\[
(\tilde S',\tilde T')
=
(\tilde S,\tilde T) +(S_{2},T_{2})\sltwo(\epsilon_{1},\epsilon_{3};\epsilon_{2},\epsilon_{4}),
\quad \epsilon_{j}=0, \text{or } 1 \ (j=1,\dots,4).
\]
Now, using properties of the Weil pairing, we have
\begin{align*}
e_{2N}(\tilde S',\tilde T')
&=e_{2N}(\tilde S+\epsilon_{1}S_{2}+\epsilon_{2}T_{2},  
\tilde T+\epsilon_{3}S_{2}+\epsilon_{4}T_{2})
\\ 
&=e_{2N}(\tilde S,\tilde T) e_{2N}(\tilde S,T_{2})^{\epsilon_{4}}e_{2N}(S_{2},\tilde T)^{\epsilon_{1}}
=(-1)^{\epsilon_{1}+\epsilon_{4}}e_{2N}(\tilde S,\tilde T).
\end{align*}
Thus, $e_{2N}(\tilde S',\tilde T')=\tilde\zeta$ if and only if $\epsilon_{1}=\epsilon_{4}$.  In other words, there are eight pairs of $(\tilde S, \tilde T)$ above $(S,T)$.  Since $(\tilde S',\tilde T')\sim (\tilde S,\tilde T)$ if and only if $\epsilon_{1}=\epsilon_{4}$ and $\epsilon_{2}=\epsilon_{3}=0$ by definition, we see that there are four different classes up to similarity according to $(\epsilon_{2},\epsilon_{3})=(0,0),(1,0),(0,1)$, or $(1,1)$.
\end{proof}

In view of Lemma~\ref{lem:level-N-2N}, we define a subgroup between $\Gamma(N)$ and $\Gamma(2N)$ as follows.

\begin{definition}\label{def:Gamma^(N)(2N)}
Define the subgroup $\Gamma^{(N)}(2N)$ of $\Gamma(N)$ by
\[
\Gamma^{(N)}(2N) =\left\{\left.\sltwo(a,b;c,d) \in SL_{2}(\Z)\,\right|\,
a\equiv d\equiv 1 \bmod N
, b\equiv c\equiv 0\bmod 2N\right\}.
\]
\end{definition}

It is easy to see that we have
\[
\Gamma(N)\vartriangleright\Gamma^{(N)}(2N)\vartriangleright\Gamma(2N).
\]
Since $N$ is assumed even, we have
\begin{gather*}
\Gamma(N)/\Gamma^{(N)}(2N)
=\left\<\sltwo(1,N;0,1), \sltwo(1,0;N,1)\right\>
\simeq \Z/2\Z\times\Z/2\Z,
\\
\Gamma^{(N)}(2N)/\Gamma(2N)
=\left\<\sltwo(1+N,0;0,1+N)\right\>
\simeq \Z/2\Z.
\end{gather*}

\begin{remark}
(1) \ 
If $N$ is odd, $\Gamma^{(N)}(2N)$ coincide with $\Gamma(2N)$, and $[\Gamma(N):\Gamma^{(N)}(2N)]=[\Gamma(N):\Gamma(2N)]=6$.

(2) \ 
The group $\Gamma^{(N)}(2N)$ coincides with the transformation group appears in Hurwitz \cite{Hurwitz} in the case of $N\equiv0\bmod 4$.  If $N\equiv 2\bmod 4$, the group in~\cite{Hurwitz} is a different group. 
For comparison between our immersion and that of Hurwitz, see Appendix~\ref{appendix}.
\end{remark}

\begin{definition}[$\Gamma^{(N)}(2N)$-structure]%
\label{def:Gamma^(N)(2N)-structure}
Fix a primitive $N$th root of unity~$\zeta$, and choose $\tilde \zeta$ such that $\tilde\zeta^{2}=\zeta$.  
Let $(S,T)$ be a $\Gamma(N)$-structure. 
By a $\Gamma^{(N)}(2N)$-structure above the given $\Gamma(N)$-structure $(S,T)$ we mean an equivalence class of $\Gamma(2N)$-structure $(\tilde S,\tilde T)$ above $(S,T)$ modulo similarity.
\end{definition}

\begin{remark}\label{rmk:def:Gamma^(N)(2N)-structure}
Sometimes a $\Gamma(N)$-structure is defined as a symplectic isomorphism $\Z/N\Z\times \mu_{N}\to E[N]$, where the anti-symmetric pairing $\< \ , \ \>$ on $\Z/N\Z\times \mu_{N}$ is defined by
\[
\<(a_{1},\zeta_{1}),(a_{2},\zeta_{2})\> = \zeta_{2}^{a_{1}}/\zeta_{1}^{a_{2}},
\quad a_{i}\in \Z/N\Z, \ \zeta_{i}\in \mu_{N} \ (i=1,2).
\]
Define a map $d:\Z/2N\Z\times \mu_{2N}\to \Z/N\Z\times \mu_{N}$ by 
\[ d: (a\bmod 2N,\zeta^b)\mapsto (a\bmod N, \zeta^{2b}).
\]
Then a $\Gamma^{(N)}(2N)$-structure above a $\Gamma(N)$-structure is nothing but a pair of symplectic isomorphisms $\phi_{N}:\Z/N\Z\times \mu_{N}\to E[N]$ and $\phi_{2N}:\Z/2N\Z\times \mu_{2N}\to E[2N]$ that makes the following diagram commutative:
\[
\begin{tikzcd}
\Z/2N\Z\times \mu_{2N} \arrow[r,"\phi_{2N}"] \arrow[d, "d"'] & {E[2N]} \arrow[d, "{[2]}"] \\
\Z/N\Z\times \mu_{N} \arrow[r,"\phi_{N}"]                   & {E[N]}                    
\end{tikzcd}
\]
\end{remark}

We now fix a $\Gamma^{(N)}(2N)$-structure $(\tilde S, \tilde T)$ above $(S,T)$. 
Define
\[
\dd_{\tilde T}=\sum\limits_{n=0}^{N-1}\bigl\{\tilde T+nT\bigr\}.
\]

\begin{lemma}\label{lem:D-tilde-T}
\begin{enumerate}
\item
There is a function $\ph_{\tilde T}$ such that $\div \ph_{\tilde T}=\dd_{\tilde T}-N\{O\}$.  As a consequence, the divisor $\dd_{\tilde T}$ is linearly equivalent to $N\{O\}$.
\item
If $\tilde T'=\tilde T + T_{2}$, then we have
\(
\dd_{\tilde T'}=\dd_{\tilde T}.
\)
\end{enumerate}
\end{lemma}

\begin{proof}
(1) \ 
Since $\addition_{E}\dd_{\tilde T}=N\tilde T+\frac{1}{2}N(N-1)T=O$, the assertion follows immediately from Abel's Theorem (Theorem~\ref{th:abel}).  The assertion (2) follows immediately from the definition.
\end{proof}

\begin{remark}
We may still use $\dd_{T}=\sum_{n=1}^{N-1}\{nT\}$ to obtain an immersion to $\P^{N-1}$ even though $\addition_{E}\dd_{T}\neq O$.  
\end{remark}

\begin{lemma}\label{lem:choice-f-even}
Let $(\tilde S, \tilde T)$ be a $\Gamma^{(N)}(2N)$-structure above $(S,T)$.  Choose a function $F_{\tilde T}$ such that $\div F_{\tilde T}=D_{\tilde T}-N\{0\}$. Then, there is a unique function $f_{\tilde S,\tilde T}$ satisfying the condition
\begin{equation}\label{eq:f_ST}
\div f_{\tilde S,\tilde T}=\action{S}\dd_{\tilde T} - \dd_{\tilde T}, 
\quad f_{\tilde S,\tilde T}(\tilde S)
=\action{S}F_{\tilde T}(\tilde S)/F_{\tilde T}(\tilde S).
\end{equation}
Moreover $(S,f_{\tilde S,\tilde T})\in E[N](\dd_{\tilde T})$ satisfies the following.
\begin{enumerate}
\item $(S,f_{\tilde S,\tilde T})^{N}=1$.
\item $\action{T_{2}}f_{\tilde S,\tilde T}(\tilde S)
=-f_{\tilde S,\tilde T}(\tilde S)$.
\end{enumerate}
\end{lemma}

\begin{proof}
Choose an element $(S,\ph)\in E[N](N\{O\})$ such that $\ph(\tilde S)=1$, and define
\(
f_{\tilde S,\tilde T}=\ph\,\action{S}F_{\tilde T}/F_{\tilde T}.
\)
Then, $f_{\tilde S,\tilde T}$ satisfies the condition \eqref{eq:f_ST}. Since $f_{\tilde S,\tilde T}$ does not depend on the choice of $F_{\tilde T}$, $f_{\tilde S,\tilde T}$ is uniquely determined.  By Lemma~\ref{lem:E[N](NO)}\,(1), we have $(S,f_{\tilde S,\tilde T})^{N}=\bigl(S,\ph\,{\action{S}F_{\tilde T}}/{F_{\tilde T}}\bigr)^{N}=(S,\ph)^{N}= \ph(\tilde S)^{N}$=1, which proves (1).  The assertion (2) is an immediate consequence of Lemma~\ref{lem:E[N](NO)}\,(2).
\end{proof}

\begin{definition}\label{def:X_k-even}
Let $f_{\tilde S,\tilde T}$ be the function defined in Lemma~\ref{lem:choice-f-even}. 
Define $X^{(\tilde S,\tilde T)}_{k}\in \L(D_{\tilde T})$ indexed by $k\in \Z/N\Z$ by 
\[
X^{(\tilde S,\tilde T)}_{0} = 1,
 \quad
X^{(\tilde S,\tilde T)}_{k}
=\tau_{1}(S,f_{\tilde S,\tilde T})X^{(\tilde S,\tilde T)}_{k-1} 
\ (k=1,\dots, N-1).
\]
\end{definition}

\begin{lemma}\label{lemma:X_k-even}
\begin{enumerate}
\item 
$X^{(\tilde S,\tilde T)}_{k}$ is well-defined, i.e., 
\(
\tau_{1}(S,f_{\tilde S,\tilde T})X^{(\tilde S,\tilde T)}_{N-1} = X^{(\tilde S,\tilde T)}_{0}.
\) 
\item 
$\div X^{(\tilde S,\tilde T)}_{k} =\action{kS}\dd_{\tilde T} - \dd_{\tilde T}$ for $k\in \Z/N\Z$.
\item
The action of $(T,1)\in E[N](\dd_{\tilde T})$ is given by
\[
\tau_{1}(T,1)X^{(\tilde S,\tilde T)}_{k}
=\chi_{S}^{k}(T)X^{(\tilde S,\tilde T)}_{k}
=e_{N}(S,T)^{k}X^{(\tilde S,\tilde T)}_{k}.
\]
\item
$X^{(\tilde S,\tilde T)}_{k}$ is a basis of $\L^{1}(\dd_{\tilde T},\chi_{S}^{k})$ for $k\in \Z/N\Z$. 
\item
$X^{(\tilde S,\tilde T)}_{k}(-P)=X^{(\tilde S,\tilde T)}_{-k}(P)$ for any $P\in E$ and $k\in \Z/N\Z$.
\end{enumerate}
\end{lemma}

\begin{proof}
The proof is the same as Lemma~\ref{lem:X_k-odd}.
\end{proof}

{\def\frac#1#2{#1\!/#2}
\begin{lemma}\label{lem:X_k-2N-str}
\begin{enumerate}
\item
$X_{k}^{(\tilde S,\tilde T)}$ is unchanged if $(\tilde S,\tilde T)$ is replaced by $(\tilde S+S_{2},\tilde T+T_{2})$, i.e.,
\[
X_{k}^{(\tilde S+S_{2},\tilde T+T_{2})}= X_{k}^{(\tilde S,\tilde T)}, \quad k=0,\dots,N-1.
\]
In other words, the function $X_{k}^{(\tilde S,\tilde T)}$ is uniquely determined by the $\Gamma^{(N)}(2N)$-struc\-ture above $(S,T)$, 
\item
If $(\tilde S,\tilde T)$ is replaced by $(\tilde S+T_{2},\tilde T)$, then
\[
X_{k}^{(\tilde S+T_{2},\tilde T)}
= (-1)^{k}X_{k}^{(\tilde S,\tilde T)}, \quad k=0,\dots,N-1.
\]
\item
If $(\tilde S,\tilde T)$ is replaced by $(\tilde S,\tilde T+S_{2})$, then
\[
X_{k}^{(\tilde S,\tilde T+S_{2})}
=X_{\frac{N}{2}+k}^{(\tilde S,\tilde T)}
/X_{\frac{N}{2}}^{(\tilde S,\tilde T)}, \quad k=0,\dots,N-1.
\]
\end{enumerate}
\end{lemma}

\begin{proof}
(1) \ 
It suffices to show that the function $f_{\tilde S,\tilde T}$ appearing in Lemma~\ref{lem:choice-f-even} is unchanged.  First, we have $D_{\tilde T+T_{2}}=D_{\tilde T}$ by Lemma~\ref{lem:D-tilde-T}\,(2).  According to the proof of Lemma~\ref{lem:choice-f-even}, $f_{\tilde S+S_{2},\tilde T}$ is defined by choosing $(S,\ph')\in E[N](N\{O\})$ such that $\ph'(\tilde S+S_{2})=1$ and letting $f_{\tilde S+S_{2}}=\ph'\,\action{S}F_{\tilde T}/F_{\tilde T}$.  But, by Lemma~\ref{lem:E[N](NO)}\,(2), $\ph'(\tilde S+S_{2})=\ph'(\tilde S)$.  This means we may use the same $\ph$ as the one we used to define $f_{\tilde S,\tilde T}$, and thus $f_{\tilde S+S_{2},\tilde T}=f_{\tilde S,\tilde T}$.

(2) \ 
By Lemma~\ref{lem:E[N](NO)}\,(2), we have $\action{T_{2}}f_{\tilde S,\tilde T}(\tilde S)=e_{2}(T_{2},S_{2})f_{\tilde S,\tilde T}(\tilde S)=-f_{\tilde S,\tilde T}(\tilde S)$.  So, the function $f=-\action{T_{2}}f_{\tilde S,\tilde T}$ satisfies $\div f = \action{S+T_{2}}\dd_{\tilde T} - \action{T_{2}}\dd_{\tilde T}=\action{S}\dd_{\tilde T} - \dd_{\tilde T}$, and $f(\tilde S+T_{2})=f_{\tilde S,\tilde T}(\tilde S)=\action{S}F_{\tilde T}(\tilde S)/F_{\tilde T}(\tilde S)$.  This implies that $-\action{T_{2}}f_{\tilde S,\tilde T}$ is nothing but the function $f_{\tilde S+T_{2},\tilde T}$.  The assertion follows from this immediately.

(3) \ 
Let $F_{\tilde T}$ be a function satisfying $\div F_{\tilde T}=\dd_{\tilde T}-N\{0\}$.  Since $\div X_{\frac{N}{2}}^{(\tilde S,\tilde T)}=\action{S_{2}}\dd_{\tilde T}-\dd_{\tilde T}=\dd_{\tilde T+S_{2}}-\dd_{\tilde T}$, we have $\div F_{\tilde T}X_{\frac{N}{2}}^{(\tilde S,\tilde T)}=\dd_{\tilde T+S_{2}}-N\{0\}$.
So, $f_{\tilde S, \tilde T+S_{2}}$ is defined by
\(
f_{\tilde S, \tilde T+S_{2}}
=f_{1}\cdot\action{S}(F_{\tilde T}X_{\frac{N}{2}}^{(\tilde S,\tilde T)})/
(F_{\tilde T}X_{\frac{N}{2}}^{(\tilde S,\tilde T)})
\)
where $f_{1}$ is a function satisfying $\div f_{1}=N\{S\}-N\{O\}$ and $f_{1}(\tilde S)=1$.  Then, we have
\[
f_{\tilde S, \tilde T+S_{2}}
=f_{1}\cdot{\action{S}F_{\tilde T}}/{F_{\tilde T}}\cdot 
{\action{S}X_{\frac{N}{2}}^{(\tilde S,\tilde T)}}
/{X_{\frac{N}{2}}^{(\tilde S,\tilde T)}}
=
{f_{\tilde S,\tilde T}\action{S}X_{\frac{N}{2}}^{(\tilde S,\tilde T)}}
/{X_{\frac{N}{2}}^{(\tilde S,\tilde T)}}
={X_{\frac{N}{2}+1}^{(\tilde S,\tilde T)}}
/{X_{\frac{N}{2}}^{(\tilde S,\tilde T)}}.
\]
This implies
\[
X_{1}^{(\tilde S,\tilde T+S_{2})}=f_{\tilde S, \tilde T+S_{2}}
=X_{\frac{N}{2}+1}^{(\tilde S,\tilde T)}/X_{\frac{N}{2}}^{(\tilde S,\tilde T)},
\]
and the assertion (2) follows by induction on~$k$.
\end{proof}

\begin{theorem}\label{thm:proj-coord-even}
Let $E$ be an elliptic curve, and $E\hookrightarrow\P^{N-1}$ an immersion as an elliptic normal curve of even degree~$N$ via the complete linear system $|N\{O\}|$.  Choose a primitive $N$th root of unity $\zeta$, and then choose $\tilde \zeta$ such that $\tilde\zeta^{2}=\zeta$.  Let $(\tilde S,\tilde T)$ be a $\Gamma^{(N)}(2N)$-structure above a $\Gamma(N)$-structure $(S,T)$. Then, $(\tilde S,\tilde T)$ determines a unique coordinate system of $\P^{N-1}$ such that the translation maps $\tau_{S}$, $\tau_{T}$, and the multiplication-by-$(-1)$ map $[-1]$ are given by the matrices $M_{T}$, $M_{S}$ and $M_{[-1]}$ in \eqref{eq:matrices}, respectively.

For a given $\Gamma(N)$-structure $(S,T)$, there are four different choices of such coordinate systems related by the change of coordinates of $\P^{N-1}$ given by the transition matrices generated by the following two\/:
\[
M_{S}^{\frac{N}{2}}=
\setlength{\arraycolsep}{4pt}
\renewcommand{\arraystretch}{0.64}
\left[
\begin{array}{*3c|*3c}
& & & & &\\
& O & & & I_{\frac{N}{2}} & \\
& & & & &\\
\hline
& & & & &\\
& I_{\frac{N}{2}} & & & O  \\
& & & & &
\end{array}
\right],
\quad
M_{T}^{\frac{N}{2}}=
\setlength{\arraycolsep}{4pt}
\renewcommand{\arraystretch}{0.72}
\left[
\begin{array}{*5r}
1 & 0 &  &   &   \\
0 & -1&  & O &  \\
  &   & \ddots &  & \\ 
  & O &  & 1 & 0 \\
  &   &  & 0& -1
\end{array}
\right], 
\]
\end{theorem}

\begin{proof}
The existence of such a coordinate system follows from Lemma~\ref{lemma:X_k-even}, and the uniqueness follows from Lemma~\ref{lem:X_k-2N-str}\,(1). The last part is a consequence of  Lemma~\ref{lem:X_k-2N-str}\,(2) and (3).
\end{proof}

When the choice of $(\tilde S,\tilde T)$ is understood, we write $X^{(\tilde S,\tilde T)}_{k}=X_{k}$ for simplicity.

\begin{proposition}\label{prop:quad-eqs-even}
The notation being as above, we have the following.
\begin{align*}
&\I^{2}(\dd_{\tilde T}) 
= \textstyle{\bigoplus_{k\in\Z/N\Z}}\I^{2}(\dd_{\tilde T},\chi_{S}^{k}),\\
&\I^{2}(\dd_{\tilde T},\chi_{S}^{k})=\I^{2}(\dd_{\tilde T})\cap \<X_{i}X_{j}\mid i+j\equiv k \mod N\>,
\\
&\dim \I^{2}(\dd_{\tilde T},\chi_{S}^{k}) =
\begin{cases}
(N-2)/2 & \text{if $k\equiv 0\bmod 2$,}\\
(N-4)/2 & \text{if $k\equiv 1\bmod 2$,}
\end{cases}
\\
&\dim \I^{2}(\dd_{\tilde T}) =N(N-3)/2.
\end{align*}
\end{proposition}

\begin{proof}
The contents of \S3.2 of \cite{Velu:1978} are still valid even if $N$ is even. In this case there are two $\chi'\in \Hom(C,\mu_{n})$ such that $\chi'^{2}=\chi_{S}^{k}$ if $k\equiv 0\bmod 2$, and there is no such $\chi'$ if $k\equiv 1\bmod 2$.  This shows $\dim \I^{2}(\dd_{\tilde T},\chi_{S}^{k})$ equals as stated by Corollaire~3.6 of~\cite{Velu:1978}.
\end{proof}

In order to find a basis of $\I^{2}(\dd_{\tilde T})$, it suffices to find a basis of $\I^{2}(\dd_{\tilde T},\chi_{0})$ and $\I^{2}(\dd_{\tilde T},\chi_{S})$, which we will do for the case $K=\C$, using theta functions.
} 

\section{Theta functions}\label{sec:theta}

From now on, we assume that the base field is the field of complex numbers~$\C$.  We construct functions $X_{0},X_{1},\dots,X_{N-1}$ defined and described in \S4 using theta functions.  Then, we find quadratic relations among them from the classical relations among theta functions.

\begin{definition}\label{def:theta}
For a pair of real numbers $(p,q)$, we define the theta function $\theta_{(p,q)}(z,\tau)$ with characteristic $(p,q)$ by
\[
\theta_{(p,q)}(z,\tau):=\sum_{n\in\Z}
\e\bigl({\tfrac{1}{2}(n+p)^2\tau + (n+p)(z+q)}\bigr),
\]
where $z\in\C$ and $\tau\in\H=\{\tau\in\C\mid \Im\tau>0\}$, and $\e(x)=e^{2\pi i x}$.  
\end{definition} 

This conforms with the definition in Mumford \cite[Ch.~I.~\S3]{Mumford:Tata-I}.  
We have the following fundamental formulas.
\begin{proposition}\label{prop:theta-basic}
Suppose $p,q,r,s\in \R$, and $l,m\in\Z$. Then, we have
\begin{enumerate}
\setlength{\itemsep}{\medskipamount}
\item $\theta_{(p,q)}(z+s,\tau) = \theta_{(p,q+s)}(z,\tau)$.
\item $\theta_{(p,q)}(z+r\tau,\tau)= 
\e(-\frac{1}{2}r^{2}\tau-rz-rq)\theta_{(p+r,q)}(z,\tau)$.
\item $\theta_{(p+l,q+m)}(z,\tau)
= \e(pm)\,\theta_{(p,q)}(z,\tau)$.
\end{enumerate}
\end{proposition}

Let $N$ be a positive integer.  Although we are interested mainly in the case where $N$ is even, we do not restrict ourselves to even $N$ until \S5.4.
\begin{definition}\label{def:theta-N}
Let $N$ be a positive integer.  For an integer or a half-integer~$k$, define 
\begin{align*}
\theta^{(N)}_k(z,\tau)
 &:=\theta_{(\frac{1}{2}-\frac{k}{N},\frac{N}{2})}(Nz,N\tau)
\\
&\phantom{:}
=\sum_{n\in\Z}\e\Bigl(
{\tfrac{1}{2}N\bigl(n-\tfrac{k}{N}+\tfrac1{2}\bigr)^2\tau
+N\bigl(n-\tfrac{k}{N}+\tfrac1{2}\bigr)\bigl(z+\tfrac12\bigr)}\Bigr).
\end{align*}
\end{definition}

It is easy to verify that $\theta^{(N)}_{k+N}(z,\tau)=
\theta^{(N)}_k(z,\tau)$, and thus $\theta^{(N)}_k(z,\tau)$ depends only on the class $k\bmod N$. 

\subsection{Basic properties of $\theta^{(N)}_{k}(z,\tau)$ as a function of~$z$}\label{subsec:basic_properties}

First, we fix a positive integer $N$ and a point $\tau\in \H$.  
\begin{proposition}\label{prop:theta-basic2}
For any positive integer $N\in \N$, and any integer or half-integer $k\in \frac{1}{2}\Z$, the following relations hold.
\begin{enumerate}%
[itemsep=\smallskipamount]
\item $\theta^{(N)}_k(z+1,\tau) = (-1)^{N+2k}\,\theta^{(N)}_k(z,\tau)$.
\item $\theta^{(N)}_k(z+\tau,\tau) 
= (-1)^{N}\e\bigl(-\frac{N}{2}\tau - Nz\bigr)\,\theta^{(N)}_k(z,\tau)$.
\item $\theta^{(N)}_k\bigl(z+\tfrac{1}{N},\tau\bigr)
= -\e\bigl(-\frac{k}{N}\bigr)\theta^{(N)}_k(z,\tau)$.
\item $\theta^{(N)}_k\bigl(z+\tfrac{\tau}{N},\tau\bigr)
= -\e\bigl(-\frac{\tau}{2N} -z\bigr)\,\theta^{(N)}_{k-1}(z,\tau)$.
\item  $\theta^{(N)}_k\bigl(z+\tfrac{\tau}{2N},\tau\bigr)
= \e\bigl(-\tfrac{\tau}{8N} -\tfrac{z}{2}-\tfrac{1}{4}\bigr)\,\theta^{(N)}_{k-\frac{1}{2}}(z,\tau)$.
\item $\theta^{(N)}_k(-z,\tau) = (-1)^{N+2k}\theta^{(N)}_{-k}(z,\tau)$.
\end{enumerate}
\end{proposition}

\begin{proof}
These formulas follow easily from the definition and Proposition~\ref{prop:theta-basic}.  
See also \cite{Mumford:Tata-I} for details.
\end{proof}

We  denote $\theta^{(N)}_k(z,\tau)$ simply by $\theta_k(z)$ if no confusion arises.

\begin{lemma}\label{lem:zeros}
Let $Z^{(N)}_{k}$ be the set of zeros of $\theta_{k}(z)$. 
\begin{enumerate} 
\item If $N$ is odd, then we have 
$Z^{(N)}_{k}=\left\{\left.\frac{m}{N}+\frac{k}{N}\tau+n\tau\,\right|\, m,\,n\in\Z\right\}$. 
\[
\includegraphics[scale=0.8]{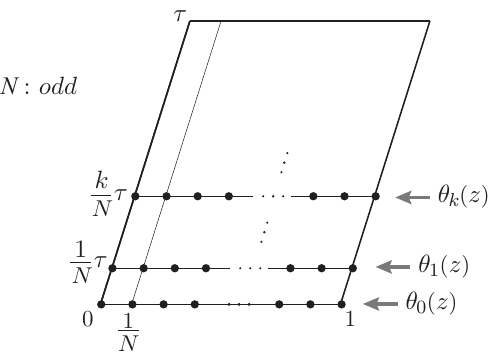}
\]
\item If $N$ is even, then 
$Z^{(N)}_{k}=\left\{\left.\frac{1}{2N}+\frac{m}{N}+\frac{k}{N}\tau+n\tau\,\right|\, m,\,n\in\Z\right\}$.
\[
\includegraphics[scale=0.8]{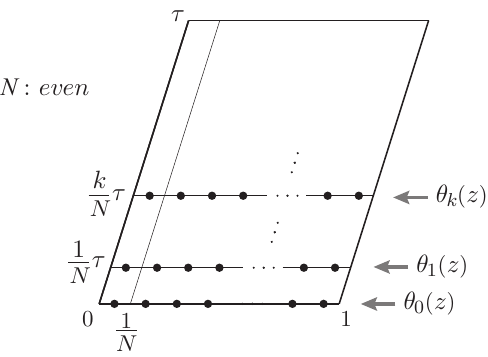}
\]
\end{enumerate}
\end{lemma}

\begin{proof}
These follow easily from the fact that the set of  zeros of $\theta_{(p,q)}(z,\tau)$ is given by $\{(l+\frac{1}{2}-p)\tau+
(m+\frac{1}{2}-q)\mid l,m\in\Z\}$.
\end{proof}

\begin{lemma}\label{lem:lin-indep}
$\theta_{0}(z), \theta_{1}(z),\dots,\theta_{N-1}(z)$ are linearly independent over $\C$.
\end{lemma}

\begin{proof}
This is easily seen by looking at the series expansions of $\theta_i(z)$'s.
\end{proof}

\begin{definition}\label{def:S-T-in-C}
For $\tau\in\H$, let $\Lambda_{\tau}=\<1,\tau\>$ be the lattice in $\C$ spanned by $1$ and~$\tau$, and let $E_{\tau}=\C/\Lambda_{\tau}$ be the elliptic curve with modulus~$\tau$.  For $N\in \N$, define points $S,T\in E_{\tau}[N]$, and $ \tilde S,\tilde T \in E_{\tau}[2N]$ by
\[
\begin{array}{ll}
S=\frac{\tau}{N} \bmod \Lambda_{\tau},&T=\frac{1}{N} \bmod \Lambda_{\tau},
\\[\medskipamount]
\tilde S=\frac{\tau}{2N} \bmod \Lambda_{\tau}, \quad &\tilde T=\frac{1}{2N} \bmod \Lambda_{\tau}.
\end{array}
\]
\end{definition}

\begin{lemma}\label{lem:weil-pairing-in-C}
Let $\zeta_{N}=\e\bigl(\tfrac{1}{N})=e^{2\pi i/N}$ and $\tilde \zeta_{N}=\zeta_{2N}=\e\bigl(\tfrac{1}{2N})=e^{2\pi i/2N}$.  Then $(S,T)$ is a level $N$ structure with $e_{N}(S,T)=\zeta_{N}$ and $(\tilde S,\tilde T)$ is a level $2N$ structure above $(S,T)$ with $\tilde \zeta_{N}^{2}=\zeta_{N}$.
\end{lemma}

\begin{proof}
Define a function $f(z)$ on $\C$ by
\[
f(z)=
\frac{\theta_{1}(z)}
{\theta_{0}(z)} 
\quad \text{if $N$ is odd,}
\quad
f(z)=\frac{\theta_{1}\bigl(z-\frac{1}{2N}\bigr)}
{\theta_{0}\bigl(z-\frac{1}{2N}\bigr)} 
\quad\text{if $N$ is even.}
\]
Then, the divisor of $f$ is given by
\[
\operatorname{div} f = \sum_{j=0}^{N-1}
\left(\bigl\{\tfrac{\tau}{N}+\tfrac{j}{N}\bigr\}
- \bigl\{\tfrac{j}{N}\bigr\}\right).
\]
Thus, by Theorem~\ref{thm:weil-pairing} and Proposition~\ref{prop:theta-basic2} (3),  we have 
\[
e_{N}(S,T) 
= e_{N}\left(\tfrac{\tau}{N}\bmod \Lambda_{\tau},
	\tfrac{1}{N}\bmod \Lambda_{\tau}\right)
=\frac{\action{T}\!f(z)}{f(z)}=\frac{f\bigl(z-\frac{1}{N}\bigr)}{f(z)}
= \zeta_{N}.
\]
By the same token we have $e_{2N}(\tilde S,\tilde T)=\zeta_{2N}=\tilde\zeta_{N}$, which implies $(\tilde S,\tilde T)$ is above $(S,T)$.
\end{proof}

\subsection{Projective immersion via theta functions}\label{subsec:projective_immersion}

\begin{theorem}\label{theorm:X-by-theta}
Let $E_{\tau}=\C/\Lambda_{\tau}$ be the elliptic curve with modulus~$\tau\in\H$, and let $S$, $T$, $\tilde S$, and $\tilde T$ be as in Definition~\ref{def:S-T-in-C}.  Define functions $X_{k}(z)$ on $\C$ indexed by $k\in\Z/N\Z$ by 
\begin{equation}\label{eq:def-X}
X_{k}(z) = \frac{\theta_{k}(z)}{\theta_{0}(z)}
=\frac{\theta^{(N)}_{k}(z,\tau)}{\theta^{(N)}_{0}(z,\tau)} \quad k\in\Z/N\Z. 
\end{equation}
Then $X_{k}(z)$ satisfy the following.
\begin{enumerate}
\item
$X_{k}$ is a function on $E_{\tau}$.
\item
$\action{T}X_{k}
=\chi_{S}^{k}(T)X_{k}
=\zeta_{N}^{k}X_{k}$.
\item
$\action{S}X_{k}
=X_{k+1}\cdot(\theta_{0}/\theta_{1})$.
\item
$X_{k}(-z)=X_{-k}(z)$ for any $z\in E_{\tau}$.
\item 
If $N$ is odd, then
\begin{enumerate}
[leftmargin=2pc]
\item 
$\div X_{k}=\action{kS}\dd_{T}-\dd_{T} $, where $\dd_{T}=\sum_{k=0}^{N-1}\{kT\}$.
\item
$X_{0},\dots,X_{N-1}$ form a basis of $\L^{1}(\dd_{T})$.
\end{enumerate}
\item 
If $N$ is even, then
\begin{enumerate}
[leftmargin=1.5pc]
\item 
$\div X_{k}=\action{kS}\dd_{\tilde T}-\dd_{\tilde T}$, where $D_{\tilde T}=\sum_{k=0}^{N-1}\{\tilde T+kT\}$.
\item
$X_{0},\dots,X_{N-1}$ form a basis of $\L^{1}(\dd_{\tilde T})$.
\end{enumerate}
\end{enumerate}
\end{theorem}

\begin{proof}
We see that $X_{k}$ are doubly periodic function with period $1$ and $\tau$ by Proposition~\ref{prop:theta-basic2} (1) and (2).  The assertions (2), (3) and (4) follow from Proposition~\ref{prop:theta-basic2} (3), (4) and (6), respectively.  The assertions (5)(i) and (6)(i) follow from Lemma~\ref{lem:zeros}.  From Lemma~\ref{lem:lin-indep}, the $X_{k}$'s are linearly independent.  Thus, if $N$ is odd, the functions $X_{0},X_{1},\dots,X_{N-1}$ form a basis of $\L^{1}(\dd_{T})$, and  if $N$ is even, they form a basis of $\L^{1}(\dd_{\tilde T})$.
\end{proof}

\begin{theorem}\label{thm:theta-immersion}
Let $N\in \N$ and $\tau\in\H$, and let $\Lambda_{\tau}=\<1,\tau\>$ be the lattice in $\C$ spanned by $1$ and~$\tau$.
\begin{enumerate}
\item
The map
\(
z \longmapsto   \bigl(\theta_{0}(z),\theta_{1}(z),
\dots,\theta_{N-1}(z)\bigr) \in \C^{N}
\)
induces an immersion $\Theta_{\tau}$ of the elliptic curve $E_{\tau}=\C/\Lambda_{\tau}$ into the projective space $\P^{N-1}${\upshape:}
\[
\begin{array}{rcccc}
\Theta_{\tau}:&E_{\tau}=\hskip-\arraycolsep&\C/\Lambda_{\tau} & \longrightarrow & \P^{N-1}
\\
&&z & \longmapsto & \bigl(\theta_{0}(z):\theta_{1}(z):
\dots:\theta_{N-1}(z)\bigr).
\end{array}
\]
\item
The image $E_{\Theta_{\tau}}=\Theta_{\tau}(E_{\tau})$ is a curve of degree~$N$ defined as the intersection of $N(N-3)/2$ quadrics in $\P^{N-1}$.
\item
Let $S$, $T$, $\tilde S$, and $\tilde S$ be as in Definition~\ref{def:S-T-in-C}.  The translations of $E_{\Theta_{\tau}}$ by $\Theta_{\tau}(S)$ and $\Theta_{\tau}(T)$ can be extended to automorphisms of $\P^{N-1}$, and expressed by the matrices $M_{S}$ and $M_{T}$ in \eqref{eq:matrices}, respectively.  The map $[-1]$ on $E_{\Theta_{\tau}}$ is also extended to an automorphism of $\P^{N-1}$, and express by the matrix $M_{[-1]}$ in~\eqref{eq:matrices}. 
\end{enumerate}
\end{theorem}

\begin{proof}
The assertion (1) is a consequence of Theorem~\ref{theorm:X-by-theta}.  The intersection between $E_{\Theta_{\tau}}$ and the hyperplane $X_{0}=0$ in $\P^{N-1}$ consists of $N$ points since the number of the zeros of $\theta_{0}(0)$ equals~$N$ by Lemma~\ref{lem:zeros}.  This means $E_{\Theta_{\tau}}$ is a curve of degree~$N$.  If $N$ is odd, the coordinate system of $E_{\Theta_{\tau}}\in \P^{N-1}$ given by $\Theta_{\tau}$ is exactly the one described in Proposition~\ref{prop:proj-coord-odd} with $\Gamma(N)$-structure $(\Theta_{\tau}(S),\Theta_{\tau}(T))$.  If $N$ is even, it is exactly the one described in Theorem~\ref{thm:proj-coord-even} with $\Gamma^{(N)}(2N)$ structure $(\Theta_{\tau}(\tilde S),\Theta_{\tau}(\tilde T))$ above $(\Theta_{\tau}(S),\Theta_{\tau}(T))$.  Thus, the formulas for the translations by $S$ and $T$, and the inversion $[-1]$ are as in \eqref{eq:matrices}, and the number of defining quadratic equations equals $N(N-3)/2$ by Proposition~\ref{prop:quad-eqs-odd} or \ref{prop:quad-eqs-even} depending on the parity of~$N$.
\end{proof}

In \S\ref{sec:quadratic_eqs}, we present explicit descriptions of the $N(N-3)/2$ quadrics in (2) (Theorems~\ref{thm:eq-N-odd} and
\ref{thm:eq-N-even}).

\subsection{Transformation formula for $\theta^{(N)}_{k}(z,\tau)$}\label{subsec:transformation_formula}

We now consider $\theta^{(N)}_{k}(z,\tau)$ as a function of $\tau$.  The matrix $\sltwo(a,b;c,d)\in SL_{2}(\Z)$ acts on the upper half plane by $\tau\mapsto \frac{a\tau+b}{c\tau+d}$.  We would like to know how this $SL_{2}(\Z)$ action affects the immersion $\Theta_{\tau}:E_{\tau}\to\P^{N-1}$.

The following transformation formula for the theta function is classical and fundamental. 
\begin{lemma}[cf. Igusa {\cite[p.85]{Igusa:theta}}]\label{lem:trans_formula}
Let $M=\sltwo(a,b;c,d)$ be in $SL_{2}(\Z)$.  Then we have
\begin{equation}\label{eq:trans_formula}
\theta_{(p,q)}\Bigl(\frac{z}{c\tau+d},\frac{a\tau+b}{c\tau+d}\Bigr)
=\kappa(M)\e(\phi_{M}(p,q))\e\bigl(\tfrac{cz^{2}}{2(cz+d)}\bigr)\sqrt{c\tau+d}\,
\theta_{(p',q')}(z,\tau),
\end{equation}
where $\kappa(M)$ is an eighth root of unity that depends on neither $\tau$ nor $(p,q)$,
\begin{align*}  \phi_{M}(p,q) &=  -\tfrac{1}{2}\bigl(abp^{2} + 2bcpq + cdq^{2} 
- bd(ap+cq)\bigr), \\
(p',q')&=\bigl(ap+cq-\tfrac{1}{2}ac, bp+dq-\tfrac{1}{2}bd\bigr),
\end{align*}
and $\sqrt{c\tau+d}$ is the principal value.
\end{lemma}

\begin{lemma}\label{lem:tans_gamma-2N}
Let $M=\sltwo(a,b;c,d)$ be in $\Gamma(2N)$.  Then we have
\[
\theta_{k}^{(N)}\Bigl(\frac{z}{c\tau+d},\frac{a\tau+b}{c\tau+d}\Bigr)
=\kappa'(M)\e\bigl(\tfrac{cNz^{2}}{2(cz+d)}\bigr)\sqrt{c\tau+d}\,
\theta_{k}^{(N)}(z,\tau),
\]
where $\kappa'(M)$ is an eighth root of unity that depends on neither $\tau$ nor~$k$.
\end{lemma}

\begin{proof}
Since $\theta_{k}^{(N)}\bigl(\frac{z}{c\tau+d},\frac{a\tau+b}{c\tau+d}\bigr)
=\theta_{(\frac{1}{2}-\frac{k}{N},\frac{N}{2})}\bigl(\frac{Nz}{(c/N)(N\tau)+d},\frac{a(N\tau)+bN}{(c/N)(N\tau)+d}\bigr)$, we apply Lemma~\ref{lem:trans_formula} with $M'=\sltwo(a,bN;c/N,d)$, and $(p,q)=(\frac{1}{2}-\frac{k}{N},\frac{N}{2})$.  Then, by Proposition~\ref{prop:theta-basic}, we have 
\begin{multline*}
\theta_{k}^{(N)}\Bigl(\frac{z}{c\tau+d},\frac{a\tau+b}{c\tau+d}\Bigr)
\\
=\kappa(M')\e\bigl(\phi_{M'}(\tfrac{1}{2}-\tfrac{k}{N},\tfrac{N}{2})\bigr)
\e\bigl((\tfrac{1}{2}-\tfrac{k}{N})(q'-\tfrac{N}{2})\bigr)
\e\bigl(\tfrac{cz^{2}}{2(cz+d)}\bigr)\sqrt{c\tau+d}\,
\theta_{k}^{(N)}(z,\tau).
\end{multline*}
We then see that $\kappa(M')\e\bigl(\phi_{M'}(\frac{1}{2}-\frac{k}{N},\frac{N}{2})\bigr)\e\bigl((\tfrac{1}{2}-\tfrac{k}{N})(q'-\tfrac{N}{2})\bigr)$ is an eighth root of unity that does not depend on $k$ as long as $\sltwo(a,b;c,d)$ is in $\Gamma(2N)$.
\end{proof}

\begin{proposition}\label{prop:trans_formula}
The following transformation formulas hold. 
\begin{enumerate}
[leftmargin=2pc,itemsep=\smallskipamount]
\item 
$\displaystyle \theta_k^{(N)}(z,\tau+1)=c\, \e\bigl(-\tfrac{k(N-k)}{2N}\bigr)\, \theta_k^{(N)}(z,\tau)$,
\item
$\displaystyle \theta_k^{(N)}\Bigl(\frac{z}{\tau},-\frac1{\tau}\Bigr)
=c'\,\e\bigl(\tfrac{z}{2}\bigr)\sqrt{\tfrac{\tau}{N}}\,
\sum_{j=0}^{N-1} \zeta_{N}^{-kj} \theta_j^{(N)}(z,\tau)$.  
\end{enumerate}
Here,  $c$ and $c'$  are constants independent of $k$, and $\zeta_N=\e(\tfrac{1}{N})$.
\end{proposition}
 
\begin{proof}
(1) \ Since $\theta_k^{(N)}(z,\tau+1)
=\theta_{(\frac{1}{2}-\frac{k}{N},\frac{N}{2})}(Nz,N\tau+N)$, we use the transformation formula \eqref{eq:trans_formula} with $M=\sltwo(1,N;0,1)$.  Then we have
\begin{align*}
\theta_k^{(N)}(z,\tau+1)&=\theta_{(\frac{1}{2}-\frac{k}{N},\frac{N}{2})}(Nz,N\tau+N)
\\
&=\kappa(M)\e\bigl(\tfrac{N}{2}(\tfrac{1}{4}-\tfrac{k^{2}}{N^{2}})\bigr)
\,\theta_{(\frac{1}{2}-\frac{k}{N},\frac{N}{2}-k)}(Nz,N\tau)
\\
&=\kappa(M)\e\bigl(\tfrac{N}{2}(\tfrac{1}{4}-\tfrac{k^{2}}{N^{2}})\bigr)
\e\bigl(-(\tfrac{1}{2}-\tfrac{k}{N})k\bigr)
\,\theta_{(\frac{1}{2}-\frac{k}{N},\frac{N}{2})}(Nz,N\tau)
\\
&=\kappa(M)\e\bigl(\tfrac{N}{8}-\tfrac{k(N-k)}{2N}\bigr)
\,\theta_{k}^{(N)}(z,\tau)
\\
&=c\,\e\bigl(-\tfrac{k(N-k)}{2N}\bigr)\, \theta_{k}^{(N)}(z,\tau).
\end{align*}
(2) \ Since $\theta_k^{(N)}(\frac{z}{\tau},-\frac{1}{\tau})
=\theta_{(\frac{1}{2}-\frac{k}{N},\frac{N}{2})}(\frac{z}{\tau/N},-\frac{1}{\tau/N})$, we let $\tau'=\tau/N$ and we use the transformation formula \eqref{eq:trans_formula} with $M=\sltwo(0,-1;1,0)$.  Then, we have $\phi_{M}(p,q)=\frac{N}{4}-\frac{k}{2}$ and $(p',q')=\left(\frac{N}{2},-\frac{1}{2}+\frac{k}{N}\right)$, and
\begin{align*}
\theta_k^{(N)}\Bigl(\frac{z}{\tau},\frac{-1}{\tau}\Bigr)
&=\theta_{(\frac{1}{2}-\frac{k}{N},\frac{N}{2})}
\Bigl(\frac{z}{\tau'},\frac{-1}{\tau'}\Bigr) \\
&=\kappa(M)\e\bigl(\tfrac{N}{4}-\tfrac{k}{2}\bigr)\e(\tfrac{z^2}{2\tau'})\sqrt{\tau'}
\,\theta_{(\frac{N}{2},-\frac{1}{2}+\frac{k}{N})}(z,\tau').
\end{align*}
We compute the series of $\theta_{(\frac{N}{2},-\frac{1}{2}+\frac{k}{N})}(z,\tau')$ by splitting it into $N$ parts:
\begin{align*}
\theta_{(\frac{N}{2},-\frac{1}{2}+\frac{k}{N})}(z,\tau')
&=\sum_{n\in\Z}\e\Bigl(
{\tfrac{1}{2}\bigl(n+\tfrac{N}{2}\bigr)^2\tfrac{\tau}{N}
+\bigl(n+\tfrac{N}{2}\bigr)\bigl(z-\tfrac{1}{2}+\tfrac{k}{N}\bigr)}\Bigr)
\\
&=\sum_{j=0}^{N-1}\biggl(
\sum_{n+j\equiv 0\bmod N}\e\Bigl(
{\tfrac{1}{2}\bigl(n+\tfrac{N}{2}\bigr)^2\tfrac{\tau}{N}
+\bigl(n+\tfrac{N}{2}\bigr)\bigl(z-\tfrac{1}{2}+\tfrac{k}{N}\bigr)}\Bigr)
\biggr).
\end{align*}
Since $n+j\equiv 0\bmod N$ if and only if $n=Nm-j$ for some $m\in\Z$, we have
\begin{align*}
&\sum_{n+j\equiv 0\bmod N}\e\Bigl(
{\tfrac{1}{2}\bigl(n+\tfrac{N}{2}\bigr)^2\tfrac{\tau}{N}
+\bigl(n+\tfrac{N}{2}\bigr)\bigl(z-\tfrac{1}{2}+\tfrac{k}{N}\bigr)}\Bigr)
\\
&\qquad=\sum_{m\in\Z}\e\Bigl(
{\tfrac{1}{2}\bigl((Nm-j)+\tfrac{N}{2}\bigr)^2\tfrac{\tau}{N}
+\bigl((Nm-j)+\tfrac{N}{2}\bigr)\bigl(z-\tfrac{1}{2}+\tfrac{k}{N}\bigr)}\Bigr)
\\
&\qquad=\sum_{m\in\Z}\e\Bigl(
\tfrac{1}{2}N\bigl(m-\tfrac{j}{N}+\tfrac{1}{2}\bigr)^2\tau
+N\bigl(m-\tfrac{j}{N}+\tfrac{1}{2}\bigr)\bigl(z+\tfrac{1}{2}\bigr)
\\
\noalign{\hfill$
+\bigl(Nm-j+\tfrac{N}{2}\bigr)(\tfrac{k}{N}-1\bigr)\Bigr)$
\quad\null}
&\qquad=\sum_{m\in\Z}\e\Bigl(
{\tfrac{1}{2}N\bigl(m-\tfrac{j}{N}+\tfrac{1}{2}\bigr)^2\tau
+N\bigl(m-\tfrac{j}{N}+\tfrac{1}{2}\bigr)\bigl(z+\tfrac{1}{2}\bigr)}\Bigr)
\e(-\tfrac{jk}{N}+\tfrac{k-N}{2})
\\
&\qquad =\e(\tfrac{k-N}{2})\zeta_{N}^{-jk}\theta_{j}^{(N)}(z,\tau)
\end{align*}
Therefore,
\begin{align*}
\theta_k^{(N)}\Bigl(\frac{z}{\tau},\frac{-1}{\tau}\Bigr)
&=\kappa(M)\e\bigl(\tfrac{N}{4}-\tfrac{k}{2}\bigr)\e(\tfrac{Nz^2}{2\tau})\sqrt{\tfrac{\tau}{N}}
\,\theta_{(\frac{N}{2},-\frac{1}{2}+\frac{k}{N})}(z,\tau')
\\
&=\kappa(M)\e\bigl(-\tfrac{N}{4}\bigr)\e\bigl(\tfrac{Nz^2}{2\tau}\bigr)\sqrt{\tfrac{\tau}{N}}\,
\sum_{j=0}^{N-1} \zeta_{N}^{-kj}\theta_{j}^{(N)}(z,\tau).
\qedhere
\end{align*}
\end{proof}

\begin{corollary}\label{cor:change-of-tau}
The action $\tau\mapsto \tau+1$ induces the change of coordinates of $\P^{N-1}$ given by the following transition matrices:
\allowdisplaybreaks
{\renewcommand{\arraystretch}{1.1}
\begin{align*}
P_{A}&=\left[
{\setlength{\arraycolsep}{2pt}
\begin{array}{r*5c}
\ 1 & 0 & \cdots & 0 &\cdots & 0 \\
0 & \zeta_{2N}^{-1(N-1)} &  \cdots & 0 &\cdots & 0\\
\vdots & \vdots &  \ddots & \vdots & & \vdots\\ 
0 & 0 & \cdots & \zeta_{2N}^{-k(N-k)} &\cdots &0\\
\vdots & \vdots &  & \vdots & \ddots &\vdots\\ 
0 & 0 &  \cdots & 0& \cdots & \zeta_{2N}^{-(N-1)\cdot1}
\end{array}
}\right]
=\operatorname{Diag}(\zeta_{2N}^{k(N-k)})_{0\le k\le N-1}.
\\
\noalign{\noindent
Similarly, $\tau\mapsto -1/\tau$ induces the change of coordinates given by:}
P_{B}&=\left[
{\setlength{\arraycolsep}{5pt}
\begin{array}{r*4c}
1 & 1 & 1 & \cdots & 1 \\
1 & \zeta_{N}^{-1} & \zeta_{N}^{-2} & \cdots & \zeta_{N}^{-(N-1)}\\
1 & \zeta_{N}^{-2} & \zeta_{N}^{-4} & \cdots & \zeta_{N}^{-2(N-1)}\\
\vdots & \vdots & \vdots & \ddots &\vdots\\ 
1 & \zeta_{N}^{-(N-1)} & \zeta_{N}^{-2(N-1)} & \cdots & \zeta_{N}^{-(N-1)^{2}}
\end{array}
}\right]
=\bigl(\zeta_{N}^{-kj}\bigr)_{0\le k,j\le N-1}.
\end{align*}
}
Here, we mean by the transition matrix from the old coordinates $(X_{0}:X_{1}:\dots:X_{N-1})$ to the new one $(X'_{0}:X'_{1}:\dots:X'_{N-1})$ the matrix $P$ that satisfies ${}^{t}\!(X'_{0},X'_{1},\dots,X'_{N-1})=P\,{}^{t}\!(X_{0},X_{1},\dots,X_{N-1})$.
\end{corollary}

\begin{proof}
Follows immediately from Proposition~\ref{prop:trans_formula}.
\end{proof}

From this Corollary, we obtain the representation 
\[
\rho:SL_{2}(\Z) \to PGL_{N}(\C)
\]
determined by
\[
\rho{\textstyle\sltwo(1,1;0,1)}=P_{A} 
\quad \text{ and } \quad
\rho{\textstyle\sltwo(0,-1;1,0)}=P_{B}. 
\]
For this representation, we have the following.

\begin{theorem}\label{thm:sl2z-action}
Let $P_{A}$ and $P_{B}$ be the matrices in Corollary~\ref{cor:change-of-tau}, and let $M_{S}$, $M_{T}$ and $M_{[-1]}$ be matrices in~\eqref{eq:matrices}.
\begin{enumerate}
\item
The representation $\rho$ factors through $SL_{2}(\Z)/\Gamma(N)$ if $N$ is odd, and
$SL_{2}(\Z)/\Gamma^{(N)}(2N)$ if $N$ is even.

\item
We have the relations
\begin{align*}
&\rho{\textstyle\sltwo(a,b;c,d)}\ M_{T}^{m}M_{S}^{n}\ 
\rho{\textstyle\sltwo(a,b;c,d)}^{-1}
\equiv M_{T}^{am+bn}M_{S}^{cm+dn}, \quad m,n\in \Z/N\Z,
\\
&\rho{\textstyle\sltwo(a,b;c,d)}\ M_{[-1]}\ 
\rho{\textstyle\sltwo(a,b;c,d)}^{-1}
\equiv M_{[-1]},
\end{align*}
where $\equiv$ means two matrices are equivalent in $PGL_{N}(\C)$. 
\end{enumerate}
\end{theorem}

\begin{proof}

(1) \ 
By Lemma~\ref{lem:tans_gamma-2N}, we see that $\ker\rho$ contains $\Gamma(2N)$. 
If $N$ is odd, then $\Gamma(N)/\Gamma^{(N)}(2N)=\Gamma(N)/\Gamma(2N)\simeq \mathfrak{S}_{3}$ (the symmetric group of degree 3) is generated by $\sltwo(1,N;0,1)$ and $\sltwo(1,0;N,1)$.  
Since $\rho{\textstyle\sltwo(1,0;1,1)}=P_{B}^{-1}P_{A}^{-1}P_{B}$, we have 
$\rho{\textstyle\sltwo(1,N;0,1)}\equiv \rho{\textstyle\sltwo(1,0;N,1)}\equiv I_{N}$ 
by straightforward calculations using Corollary~\ref{cor:change-of-tau}.  Hence we conclude that $\ker\rho$ contains $\Gamma(N)$.

If $N$ is even, however, we have
$\rho{\textstyle\sltwo(1,N;0,1)}\equiv M_{T}^{\frac{N}{2}}$ and
$\rho{\textstyle\sltwo(1,0;N,1)}\equiv M_{S}^{\frac{N}{2}}$, 
where $M_{T}^{\frac{N}{2}}$ and $M_{S}^{\frac{N}{2}}$ are matrices appearing in Theorem~\ref{thm:proj-coord-even}.  
By a simple calculation, we have
\[
\begin{pmatrix} 1 & N \\ 0 & 1 \end{pmatrix}
\begin{pmatrix} 1 & 0 \\ N-1 & 1 \end{pmatrix}
\begin{pmatrix} 1 & N \\ 0 & 1 \end{pmatrix}
\begin{pmatrix} 1 & 0 \\ 1 & 1 \end{pmatrix}
\equiv
\begin{pmatrix} 1+N & 0 \\ 0 & 1+N \end{pmatrix} \mod 2N,
\]
if $N$ is even.  Again, by straightforward calculations, we have
\[
P_{A}^{N}(P_{B}^{-1}P_{A}^{-1}P_{B})^{N-1}P_{A}^{N}(P_{B}^{-1}P_{A}^{-1}P_{B})
\equiv I_{N}.
\]
This proves $\rho{\textstyle\sltwo(1+N,0;0,1+N)\equiv I_{N}}$, and 
we see that $\ker\rho$ contains $\Gamma^{(N)}(2N)$.

(2) \ By direct calculations, we have
\[
\begin{alignedat}{2}
&P_{A}\ M_{T}\ P_{A}^{-1}\equiv M_{T}, \quad &
&P_{A}\ M_{S}\ P_{A}^{-1}\equiv M_{S}M_{T}, 
\\
&P_{B}\ M_{T}\ P_{B}^{-1}\equiv M_{S}, \quad &
&P_{B}\ M_{S}\ P_{B}^{-1}\equiv M_{T}^{-1}, 
\\
\end{alignedat}
\]
Since $M_{T}M_{S}\equiv M_{S}M_{T}$ in $PGL_{N}(\C)$, we obtain the first relation in the statement.  As for the second relation, it suffices to verify
$P_{A}\ M_{[-1]}\ P_{A}^{-1}\equiv M_{[-1]}$ and
$P_{B}\ M_{[-1]}\ P_{B}^{-1}\equiv M_{[-1]}$,
which can be done easily.

\end{proof}

\subsection{Modular curves}\label{subsec:modular_curves}

In this paragraph we assume $N$ is even.  Let $Y^{(N)}(2N)=\Gamma^{(N)}(2N)\backslash\H$ be the modular curve associated with $\Gamma^{(N)}(2N)$.
Let $\E^{(N)}(2N)$ be the universal elliptic curve over $Y^{(N)}(2N)$, that is, the fibration $E^{(N)}(2N)\to Y^{(N)}(2N)$ such that the fiber at $\tau\in Y^{(N)}(2N)$ is the elliptic curve $E_{\tau}=\C/\Lambda_{\tau}$. 
By Theorem~\ref{thm:theta-immersion}, we have an immersion 
\[
\begin{array}{rccc}
\Theta:&E^{(N)}(2N) &\lhook\joinrel\longrightarrow &\P^{N-1}\times Y^{(N)}(2N)
\\& (z\bmod\Lambda_{\tau},\tau) &\longmapsto &(\Theta_{\tau}(z),\tau).
\end{array}
\]
Let $o$ be the $0$-section $Y^{(N)}(2N)\to E^{(N)}(2N)$.  We obtain a morphism $\alpha:Y^{(N)}(2N)\to \P^{N-1}$ by the following diagram
\[
\begin{tikzcd}
E^{(N)}(2N) \arrow[r, "\Theta"]                & \P^{N-1}\times Y^{(N)}(2N) \arrow[d, "\pi_{1}"] \\
Y^{(N)}(2N) \arrow[u, "o"] \arrow[r, "\alpha"] & \P^{N-1}.                                     
\end{tikzcd}
\]
where $\pi_{1}:\P^{N-1}\times Y^{(N)}(2N)\to \P^{N-1}$ is the projection onto the first factor.  By taking its closure we also obtain $X^{(N)}(2N)=\overline{Y^{(N)}(2N)}\to \P^{N-1}$.  In terms of coordinates, the morphism $\alpha$ is given by $\tau\mapsto (a_0^{(N)}(\tau),\dots,a_{N-1}^{(N)}(\tau))$, where $a_k^{(N)}(\tau)$ is the ``Theta Null Werte" defined by
\[ 
a_k^{(N)}(\tau):=\theta_k^{(N)}(0,\tau)=\theta_{(\frac12-\frac{k}N,\frac{N}2)}(0,N\tau)\quad (k\in\Z).
\]
We will see later in \S\ref{sec:quadratic_eqs} that the image in $\P^{N-1}$ satisfies a set of quartic equations.  It is in general a difficult question whether these equations define $X^{(N)}(2N)$.  We will discuss this question in each of examples $N=4,6$, and $8$ later.

By using theta transformation formula (Lemma~\ref{lem:trans_formula}), we can describe
the transformation properties of $a_{k}^{(N)}(\tau)$ 
under the action of the group $\Gamma(N)/\Gamma^{(N)}(2N)$. 
By Proposition~\ref{prop:theta-basic2}~(6), we have
\[
a_{k}^{(N)}(\tau) = (-1)^N a_{-k}^{(N)}(\tau) = a_{N-k}^{(N)}(\tau).
\]
(Remember that $N$ is even.)
This implies that the image of $\alpha$ is of the form
\begin{equation}\label{eq:alpha(tau)}
(a_{0}(\tau):a_{1}(\tau):\dots:a_{\frac{N}{2}-1}(\tau):a_{\frac{N}{2}}(\tau):a_{\frac{N}{2}-1}(\tau):\dots:a_{1}(\tau)).
\end{equation}
In other words, it is contained in the linear subspace $H$ in $\P^{N-1}$ defined by the equations $X_{k}=X_{N-k}$, $k=1,\dots, \tfrac{N}{2}-1$.  The linear space $H$ is stable under the representation $\rho:SL_{2}(\Z)\to PGL_{N}(\C)$ in Theorem~\ref{thm:sl2z-action}.  Moreover, since $\rho(-I_{2})=M_{[-1]}$ acts trivially on $H$, and $\ker\rho\supset\Gamma^{(N)}(2N)$, $\rho$ induces a representation 
\[
\bar\rho: SL_{2}(\Z)/\Gamma^{(N)}(2N) \to PGL_{\frac{N}{2}+1}(\C).
\]  
Define the coordinates $(\bar X_{0}:\bar X_{1}:\dots:\bar X_{\frac{N}{2}})$  of $H$ by
\[
\renewcommand{\arraystretch}{1.1}
\left\{
\begin{array}{l}
\bar X_{0} = X_{0}, \\ 
\bar X_{1} = X_{1}+X_{N-1}, \\
\qquad \dots  \\
\bar X_{\frac{N}{2}-1} = X_{\frac{N}{2}-1}+X_{\frac{N}{2}+1},\\
\bar X_{\frac{N}{2}} = X_{\frac{N}{2}}.
\end{array}\right.
\]
Then, $\bar\rho$ is given by 
\begin{align*}
\bar\rho\sltwo(0,-1;1,0)
&=\left[\renewcommand{\arraystretch}{1.1}
\setlength{\arraycolsep}{3pt}
\begin{array}{*5c}
1 & 1 & 1 &\cdots & 1 \\
2 & \zeta+\zeta^{-1} & \zeta^{2}+\zeta^{-2} & \cdots & -2\\
2 & \zeta^{2}+\zeta^{-2} & \zeta^{4}+\zeta^{-4} &\cdots & -2\\
\vdots & \vdots & \vdots & \ddots &\vdots\\ 
1 & -1 & 1 & \cdots & (-1)^{\frac{N}{2}}
\end{array}\right], 
\\
\bar\rho\sltwo(1,1;0,1)
&=\setlength{\arraycolsep}{1pt}
\left[\begin{array}{*5c}
1 &  &  &  & \\
  & \tilde\zeta^{1(N-1)} &  &  & \\
  &  & \tilde\zeta^{2(N-2)} &  & \\
  &  &  & \ddots &  \\ 
  &  &  &  & \tilde\zeta^{N^{2}/4}\\
\end{array}\right].
\end{align*}
In particular, we have
\[
\bar\rho\sltwo(1,N;0,1)
=
\left[\renewcommand{\arraystretch}{0.66}
\setlength{\arraycolsep}{2pt}
\begin{array}{*5c}
1 &  &  &  & \\
  & -1  &  & \\
  &  & 1 &  &\\
  &  &  & \ddots  &\\ 
  &  &  &  &(-1)^{\frac{N}{2}} 
\end{array}\right],
\quad
\bar\rho\sltwo(1,0;N,1)
=
\left[\renewcommand{\arraystretch}{0.76}
\setlength{\arraycolsep}{4pt}
\begin{array}{*5c}
  &  &  &   & 1 \\
  &  &  & 1 &   \\
  &  & \udots &  &  \\
  & 1 &  &  &  \\
1 &   &  &  & 
\end{array}\right].
\]

Thus, we have the following.

\begin{proposition}  Let $N$ be an even positive integer.  Write $a^{(N)}_{k}(\tau)=a_{k}$ for short.
Then, the ratio $(a_{0}:a_{1}:\dots:a_{\frac{N}{2}})$ is invariant under the action of 
$\Gamma^{(N)}(2N)$.  Moreover, the following holds.
\smallskip
\begin{enumerate}%
[itemsep=\smallskipamount]
\item
The action $\tau\mapsto \tsltwo(1,N;0,1)\tau=\tau + N$ induces the action 
\[
(a_{0}:a_{1}:\dots:a_{k}:\dots:a_{\frac{N}{2}}) \mapsto 
(a_{0}:-a_{1}:\dots:(-1)^{k}a_{k}:\dots:(-1)^{\frac{N}{2}}a_{\frac{N}{2}}).
\]
\item
The action $\tau\mapsto \tsltwo(1,0;N,1)\tau=\frac{\tau}{N\tau + 1}$ induces the action 
\[
(a_{0}:a_{1}:\dots:a_{k}:\dots:a_{\frac{N}{2}}) \mapsto 
(a_{\frac{N}{2}}:a_{\frac{N}{2}-1}:\dots:a_{\frac{N}{2}-k}:\dots:a_{0}).
\]
\end{enumerate}
\end{proposition}

\section{Quadratic equations satisfied by theta functions}
\label{sec:quadratic_eqs}

Let $E$ be an elliptic curve over a field $K$ contained in $\C$.  If we immerse $E$ via the complete linear system $|N\{O_{E}\}|$ for some integer $N$, there exists a coordinate system $(X_{0}:X_{1}:\cdots:X_{N-1})\in\P^{N-1}$ described in Proposition~\ref{prop:proj-coord-odd} or Theorem~\ref{thm:proj-coord-even}.  Then, by a suitable choice of $\tau\in\H$, we have a map $\C/\Lambda_{\tau}\to E\subset\P^{N-1}$ given by $X_{k}=\theta_{k}^{(N)}(z,\tau)$.  In order to describe $N(N-3)/2$ quadratic equations satisfied by $E$ explicitly, we would like to find relations satisfied among $\theta_{k}^{(N)}(z,\tau)$.  It turns out that the situation is quite different depending on the parity of $N$.

We use the method of Jacobi \cite{Jacobi:Theta}, which is completely elementary and algebraic
(i.e., no function theory is used). 
\begin{definition}\label{def:Jacobi-basic}
Jacobi's basic theta functions $\vartheta_i(z)$ ($i=0,1,2,3$) are 
defined by the following formulas:
\begin{alignat*}{2}
\vartheta_0(z)&=\theta_{(0,\frac12)}(z,\tau)
&&=\sum_{n\in\Z} \e\bigl({\tfrac{1}{2}n^2\tau + n(z+\tfrac12)}\bigr),\\
\vartheta_1(z)&=\theta_{(\frac12,\frac12)}(z,\tau)
&&=\sum_{n\in\Z} \e\bigl({\tfrac{1}{2}(n+\tfrac12)^2\tau + (n+\tfrac12)(z+\tfrac12)}\bigr),\\
\vartheta_2(z)&=\theta_{(\frac12,0)}(z,\tau)
&&=\sum_{n\in\Z} \e\bigl({\tfrac{1}{2}(n+\tfrac12)^2\tau + (n+\tfrac12)z}\bigr),\\
\vartheta_3(z)&=\theta_{(0,0)}(z,\tau)
&&=\sum_{n\in\Z} \e\bigl({\tfrac{1}{2}n^2\tau + nz}\bigr).
\end{alignat*}
\end{definition}

Note that our definition differs slightly from Jacobi's original notation by some rescaling and sign.  Although we should write $\vartheta_{i}(z,\tau)$ instead of $\vartheta_{i}(z)$, we omit $\tau$ for simplicity.

The function $\vartheta_1(z)$ is an odd function and the others are even functions:
\begin{equation}\label{eq:theta-odd-even}
\begin{array}{ll}
\vartheta_0(-z)=\vartheta_0(z), \ &
\vartheta_1(-z)=-\vartheta_1(z), \\
\vartheta_2(-z)=\vartheta_2(z), \ &
\vartheta_3(-z)=\vartheta_3(z). 
\end{array}
\end{equation}
The following formulas are immediate from the definition.
\begin{equation}\label{jacobitheta}
\begin{aligned}
&\vartheta_0(z+\tfrac12)=\vartheta_3(z),\quad 
\vartheta_1(z+\tfrac12)=-\vartheta_2(z),\\
&\vartheta_2(z+\tfrac12)=\vartheta_1(z),\quad
\vartheta_3(z+\tfrac12)=\vartheta_0(z).
\end{aligned}
\end{equation}
The starting point is the identity (A)-(1) in  Jacobi \cite[p.507]{Jacobi:Theta}.  

For independent variables $w,x,y,z$, define variables $w',x',y',z'$ by
\begin{equation}\label{eq:xyzw}
\left\{\begin{aligned}
w' &=\tfrac{1}{2}(w+x+y+z), \\
x' &=\tfrac{1}{2}(w+x-y-z), \\
y' &=\tfrac{1}{2}(w-x+y-z), \\
z' &=\tfrac{1}{2}(w-x-y+z),
\end{aligned}\right.
\quad\text{or}\quad
\setlength{\arraycolsep}{2pt}
\left(\begin{array}{c}
w' \\ x' \\ y' \\ z'
\end{array}\right)
=
\frac{1}{2}
\left(\begin{array}{rrrr}
1 & 1 & 1 & 1 \\ 
1 & 1 & -1 & -1 \\ 
1 & -1 & 1 & -1 \\ 
1 & -1 & -1 & 1
\end{array}\right)
\left(\begin{array}{c}
w \\ x \\ y \\ z
\end{array}\right).
\end{equation}
Denote by $A$ the $4\times 4$ matrix (including the factor $1/2$) in the second equation of \eqref{eq:xyzw}. Then, we have 
$A\in O(4)$, $A^{2}=I$, and $\det A=-1$.
This shows that the transformation given by \eqref{eq:xyzw} induces an involution on the set 
\[ \{(w,x,y,z)\in \Z^4\,\mid\,w\equiv x\equiv y\equiv z\!\pmod2\} \]
and preserving the norm $w^2+x^2+y^2+z^2$.

\begin{proposition}\label{prop:Jacobi-A-1}
Let $w,x,y,z$ be independent variables, and let $w',x',y',z'$ be the variables defined by \eqref{eq:xyzw}.
Then, the following identity holds:
\begin{equation}\label{eq:Jacobi4}
\begin{aligned}
\vartheta_{3}(w)\vartheta_{3}(x)&\vartheta_{3}(y)\vartheta_{3}(z)
+\vartheta_2(w)\vartheta_2(x)\vartheta_2(y)\vartheta_2(z)
\\
&= \vartheta_{3}(w')\vartheta_{3}(x')\vartheta_{3}(y')\vartheta_{3}(z')
+\vartheta_2(w')\vartheta_2(x')\vartheta_2(y')\vartheta_2(z'),
\end{aligned}
\end{equation}
\end{proposition}

\begin{proof}
Jacobi's identity \eqref{eq:Jacobi4} follows from the properties of \eqref{eq:xyzw}. The reader is encouraged to consult the beautiful, original article of Jacobi \cite{Jacobi:Theta}. 
\end{proof}

\begin{theorem}\label{th:three-term}
Let $w$, $x$, $y$, $z$ be independent variables.  As in Proposition~\ref{prop:Jacobi-A-1}, define $w'$, $x'$, $y'$, $z'$ by \eqref{eq:xyzw}.  Furthermore, define variables $w'',x'',y'',z''$ by
\begin{equation}\label{eq:xyzw3}
{}^{t}\!(w'', x'', y'', z'') = A\, {}^{t}\!(w,x,y,-z).
\end{equation}
Then, the following three-term identity holds:
\begin{multline}\label{eq:threeterm}
\vartheta_1(w)\vartheta_1(x)\vartheta_1(y)\vartheta_1(z)
\\
+\vartheta_{0}(w')\vartheta_{0}(x')\vartheta_{0}(y')\vartheta_{0}(z')
\\
-\vartheta_0(w'')\vartheta_0(x'')\vartheta_0(y'')\vartheta_0(z'')=0.
\end{multline}
\end{theorem}

\begin{proof}
Replacing $w$ by $w+1$ in \eqref{eq:Jacobi4} and using \eqref{jacobitheta}, we obtain
\begin{multline}\label{eq:Jacobi2}
\vartheta_{3}(w)\vartheta_{3}(x)\vartheta_{3}(y)\vartheta_{3}(z)
-\vartheta_2(w)\vartheta_2(x)\vartheta_2(y)\vartheta_2(z)
\\
= \vartheta_{0}(w')\vartheta_{0}(x')\vartheta_{0}(y')\vartheta_{0}(z')
+\vartheta_1(w')\vartheta_1(x')\vartheta_1(y')\vartheta_1(z').
\end{multline}
Adding \eqref{eq:Jacobi4} and \eqref{eq:Jacobi2}, we obtain
\begin{multline*}
2\vartheta_3(w)\vartheta_3(x)\vartheta_3(y)\vartheta_3(z)\\
=\vartheta_{3}(w')\vartheta_{3}(x')\vartheta_{3}(y')\vartheta_{3}(z')
+\vartheta_2(w')\vartheta_2(x')\vartheta_2(y')\vartheta_2(z')\\
+\vartheta_{0}(w')\vartheta_{0}(x')\vartheta_{0}(y')\vartheta_{0}(z')
+\vartheta_1(w')\vartheta_1(x')\vartheta_1(y')\vartheta_1(z').
\end{multline*}
Replace $w,x,y,z$ by $w+\tfrac12,x+\tfrac12,y+\tfrac12,z+\tfrac12$ in this identity.  Then $w'$ becomes $w'+1$, and $x',y',z'$ unchanged.  By \eqref{jacobitheta} we obtain
\begin{multline}\label{eq:th0}
2\vartheta_0(w)\vartheta_0(x)\vartheta_0(y)\vartheta_0(z)\\
=\vartheta_{0}(w')\vartheta_{0}(x')\vartheta_{0}(y')\vartheta_{0}(z')
-\vartheta_1(w')\vartheta_1(x')\vartheta_1(y')\vartheta_1(z')\\
-\vartheta_2(w')\vartheta_2(x')\vartheta_2(y')\vartheta_2(z')
+\vartheta_{3}(w')\vartheta_{3}(x')\vartheta_{3}(y')\vartheta_{3}(z').
\end{multline}

Since the relation between $w,x,y,z$ and $w',x',y',z'$ are symmetric, we have 
\begin{multline}\label{eq:th0-inv}
2\vartheta_0(w')\vartheta_0(x')\vartheta_0(y')\vartheta_0(z')\\
=\vartheta_{0}(w)\vartheta_{0}(x)\vartheta_{0}(y)\vartheta_{0}(z)
-\vartheta_1(w)\vartheta_1(x)\vartheta_1(y)\vartheta_1(z)\\
-\vartheta_2(w)\vartheta_2(x)\vartheta_2(y)\vartheta_2(z)
+\vartheta_{3}(w)\vartheta_{3}(x)\vartheta_{3}(y)\vartheta_{3}(z),
\end{multline}
By the definition of the transformation \eqref{eq:xyzw3}, the identity \eqref{eq:th0-inv} can be translated to
\begin{multline}\label{eq:th0-inv-2}
2\vartheta_0(w'')\vartheta_0(x'')\vartheta_0(y'')\vartheta_0(z'')\\
=\vartheta_{0}(w)\vartheta_{0}(x)\vartheta_{0}(y)\vartheta_{0}(-z)
-\vartheta_1(w)\vartheta_1(x)\vartheta_1(y)\vartheta_1(-z)\\
-\vartheta_2(w)\vartheta_2(x)\vartheta_2(y)\vartheta_2(-z)
+\vartheta_{3}(w)\vartheta_{3}(x)\vartheta_{3}(y)\vartheta_{3}(-z).
\end{multline}
Now, calculating \eqref{eq:th0-inv} $-$ \eqref{eq:th0-inv-2}, and using \eqref{eq:theta-odd-even}, we obtain the relation \eqref{eq:threeterm}.
\end{proof}

\subsection{Odd case}\label{subsec:odd_case}
Assume $N$ is odd.  

From our definition of $\theta^{(N)}_{k}(z,\tau)$ and Proposition~\ref{prop:theta-basic}, we have the relations
\[ 
\theta^{(N)}_{0}(z,\tau)=  (-1)^{\frac{N-1}{2}}\vartheta_{1}(Nz,N\tau)\quad\text{and}\quad
\theta^{(N)}_{\frac{N}{2}}(z,\tau)= \vartheta_{0}(Nz,N\tau).
\]
Dropping the superscript ``$(N)$'', \eqref{eq:threeterm} is translated to
\begin{multline}\label{eq:finalthreeterm-odd}
\theta_0(w)\theta_0(x)\theta_0(y)\theta_0(z)
\\
+\theta_{\frac{N}2}(w')\theta_{\frac{N}2}(x')\theta_{\frac{N}2}(y')\theta_{\frac{N}2}(z')
-\theta_{\frac{N}2}(w'')\theta_{\frac{N}2}(x'')\theta_{\frac{N}2}(y'')\theta_{\frac{N}2}(z'')=0.
\end{multline}

We use this equation to obtain our quadratic equations.

\begin{theorem}\label{thm:eq-N-odd}
Suppose $N$ is an odd integer.  Let $E_{\Theta_{\tau}}$ be the image of $\Theta_{\tau}$ in Theorem~\ref{thm:theta-immersion}.  Let $V=H^{0}(\P^{N-1},\mathcal{I}_{E_{\Theta_{\tau}}}(2))$, and $V_{k}=V\cap \<X_{i}X_{j}\mid i+j\equiv k \mod N\>$, where $X_{i}$ $(i=0,\dots,N-1)$ are the coordinate functions of $\P^{N-1}$ in Proposition~\ref{prop:proj-coord-odd}. Let $a_{i}=\theta^{(N)}_{i}(0,\tau)$.  Then, 
\begin{enumerate}
\item
The vector space $V_{k}$ is isomorphic to $V_{0}$ for all $k=0,\dots,N-1$ by the map $X_{i}X_{j}\mapsto X_{i+k/2}X_{j+k/2}$, where indices are taken as elements of $\Z/N\Z$.
\item 
The quadratic forms
\begin{equation}\label{eq:deg2-N-odd}
a_{j+1}a_{N-j}X_{0}^{2}
-a_{\frac{N-1}{2}-j}a_{\frac{N+1}{2}+j}
X_{\frac{N+1}{2}}X_{\frac{N-1}{2}}
+a_{\frac{N-1}{2}}a_{\frac{N+1}{2}}X_{\frac{N-1}{2}-j}
X_{\frac{N+1}{2}+j}
\end{equation}
with $j=1,\ldots,\tfrac{N-3}{2}$ form a basis of the vector space $V_{0}$.
\end{enumerate}
\end{theorem}

\begin{proof} 
(1) \ 
Since $N$ is odd, $2$ is invertible in $\Z/N\Z$ and thus $k/2$ in the indices makes sense.  The isomorphism is nothing but the translation map $\tau_{S}^{k/2}$ in Proposition~\ref{prop:proj-coord-odd}.

(2) \
First, we show that if we replace $X_{i}$ by $\theta^{(N)}_{i}(z,\tau)$, the quadratic forms in \eqref{eq:deg2-N-odd} vanishes.  In \eqref{eq:finalthreeterm-odd},  let 
\[
(w,x,y,z) =\;(z,z,\tfrac{j\tau}{N},-\tfrac{(j+1)\tau}{N}\bigr).
\]
Then, we have 
\begin{align*}
(w',x',y',z') &= \bigl(z-\tfrac{\tau}{2N},z+\tfrac{\tau}{2N},\tfrac{(2j+1)\tau}{2N},-\tfrac{(2j+1)\tau}{2N}\bigr),
\\ 
(w'',x'',y'',z'') &=\bigl(z+\tfrac{(2j+1)\tau}{2N},z-\tfrac{(2j+1)\tau}{2N},-\tfrac{\tau}{2N},\tfrac{\tau}{2N}\bigr).
\end{align*}
Then, use Proposition~\ref{prop:theta-basic2}~(4) to see that \eqref{eq:deg2-N-odd} is satisfied by $X_{i}=\theta^{(N)}_{i}(z,\tau)$ and $a_{i}=\theta^{(N)}_{i}(0,\tau)$.  Recall that $a_k$ depends 
only on $k \bmod N$, and we have $a_0=0$ and $a_k=-a_{N-k}$. Thus, we may restrict the range of $j$ to $1\le j\le (N-3)/2$.  ($j=(N-1)/2$ gives the trivial relation.)

Next we show that these $(N-3)/2$ quadratic forms are independent.  Indeed, the coefficients of the terms $X_{\frac{N-1}{2}-j}X_{\frac{N+1}{2}+j}$ for $1\le j\le (N-3)/2$ are all equal to $a_{\frac{N-1}{2}}a_{\frac{N+1}{2}}$, which is nonzero, and the terms $X_{0}^{2}, X_{\frac{N+1}{2}}X_{\frac{N-1}{2}}$ are different from $X_{\frac{N-1}{2}-j}X_{\frac{N+1}{2}+j}$ for $1\le j\le (N-3)/2$.  Therefore, these $(N-3)/2$ quadratic forms are independent. 
Since $\dim V_{0}=(N-3)/2$ by Propositions~\ref{prop:quad-eqs-odd}, these  $(N-3)/2$ quadratic forms are basis of~$V_{0}$.
\end{proof}

\begin{remark}\label{rmk:quatic_eqs}
By specializing $X_i$ at $z=0$ in \eqref{eq:deg2-N-odd}, or sending $X_i\mapsto a_i$, we obtain quartic relations in $a_{j}$'s.  These quartic equations are satisfied by the image of the morphism $\alpha:Y(N)\to \P^{N-1}$ in \S\ref{subsec:modular_curves}.
\end{remark}

We briefly describe below the classical examples $N=5$ and 7.

\begin{example} When $N=5$, \eqref{eq:deg2-N-odd} yields $(5-3)/2=1$ quadratic form, which is 
\[
a_1a_2X_0^2-a_1^2X_{2}X_{3}+a_2^2X_{1}X_{4}.
\]
$V=H^{0}(\P^{N-1},\mathcal{I}_{E_{\Theta_{\tau}}}(2))$ is generated by this form and its various permutations by $M_{S}$ in Proposition~\ref{prop:proj-coord-odd}. If we define $\phi(\tau)$ by
\[ 
\phi(\tau)=-\frac{a_1(\tau)}{a_2(\tau)}
=q^{\frac{1}{5}} - q^{\frac{6}{5}} + q^{\frac{11}{5}} - q^{\frac{21}{5}} + q^{\frac{26}{5}} - q^{\frac{31}{5}} + \cdots\quad 
(q=\e(\tau)=e^{2\pi i \tau}), 
\]
we have the equations 
\[ X_i^2+\phi(\tau)X_{i+2}X_{i-2}-\frac1{\phi(\tau)}X_{i+1}X_{i-1}=0  
\quad (i=0,\ldots,4).
\]
This set of equations is the well-known Bianchi normal form \cite{Bianchi}.  
\end{example}

\begin{example}  When $N=7$, quadratic forms \eqref{eq:deg2-N-odd} give two equations 
\begin{align*}
&a_1a_2X_0^2-a_2^2X_{3}X_{4}+a_3^2 X_{2}X_{5}=0,\\
&a_2a_3X_0^2-a_1^2X_{3}X_{4}+a_3^2 X_{1}X_{6}=0.
\end{align*}
By the specialization $z=0$ in $X_i$ after a suitable shift using Proposition~\ref{prop:theta-basic2}~(4), that is, 
sending $X_i\mapsto a_{i+j}$ for some $j$, we obtain the unique relation
\[a_1^3a_2=a_2^3a_3+a_1a_3^3. \]
This is the renowned Klein's quartic, which is a model of the modular curve X(7) of genus~$3$.
\end{example}

\subsection{Even case}\label{subsec:even_case}

Next suppose $N$ is even.  
Again we start with Jacobi's \eqref{eq:Jacobi4}, which in terms of our theta's is written as
\begin{equation}\label{eq:start-even}
\begin{aligned}
\theta_{\frac{N}{2}}(w)\theta_{\frac{N}{2}}(x)&\theta_{\frac{N}{2}}(y)\theta_{\frac{N}{2}}(z)
+\theta_{0}(w)\theta_{0}(x)\theta_{0}(y)\theta_{0}(z)
\\
&= \theta_{\frac{N}{2}}(w')\theta_{\frac{N}{2}}(x')\theta_{\frac{N}{2}}(y')\theta_{\frac{N}{2}}(z')
+\theta_{0}(w')\theta_{0}(x')\theta_{0}(y')\theta_{0}(z').
\end{aligned}
\end{equation}
We can deduce our quadratic relations directly from this with the following specializations.  
For  an integer $j$, consider the four substitutions

\begin{align*}
 (w,x,y,z)=  
&\left(z,-\tfrac{j\tau}{N}, -\tfrac{j\tau}{N}, z\right),\quad
\left(-\tfrac{2j\tau}{N},0, z, z\right),\\
&\left(z-\tfrac{\tau}{N},-\tfrac{j\tau}{N}, -\tfrac{(j+1)\tau}{N}, z\right),\quad
\left(-\tfrac{(2j+1)\tau}{N},0,z, z-\tfrac{\tau}{N}\right),
\end{align*}
which, respectively, yield
\begin{multline*}
 (w',x',y',z')= 
\left(z-\tfrac{j\tau}{N},0,0, z+\tfrac{j\tau}{N}\right),\
\left(z-\tfrac{j\tau}{N},-z-\tfrac{j\tau}{N}, -\tfrac{j\tau}{N},-\tfrac{j\tau}{N}\right),\\
\left(z-\tfrac{(j+1)\tau}{N},0, -\tfrac{\tau}{N}, z+\tfrac{j\tau}{N}\right),\ 
\left(z-\tfrac{(j+1)\tau}{N},-z-\tfrac{j\tau}{N}, -\tfrac{j\tau}{N},-\tfrac{(j+1)\tau}{N},\right).
\end{multline*}
Applying these to~\eqref{eq:start-even} and using Proposition~\ref{prop:theta-basic2},
we obtain the following set of equations.
{\renewcommand{\arraystretch}{1.4}
\setlength{\arraycolsep}{3pt}
\begin{gather}
\left\{
\begin{array}{rcrcrcr}
a_{j}^{2}\,X_{0}^{2} & + & a_{\frac{N}{2}+j}^{2}\,X_{\frac{N}{2}}^{2} & = &
a_{0}^{2}\,X_{j}X_{N-j} & + & a_{\frac{N}{2}}^{2}\,X_{\frac{N}{2}+j}X_{\frac{N}{2}-j},
\\
a_{0}a_{2j}\,X_{0}^{2} & + & a_{\frac{N}{2}}a_{\frac{N}{2}+2j}\,X_{\frac{N}{2}}^{2} & = &
a_{j}^{2}\,X_{j}X_{N-j} & + & a_{\frac{N}{2}+j}^{2}\,X_{\frac{N}{2}+j}X_{\frac{N}{2}-j},
\end{array}  \label{eq:eq-V0}
\right.
\\[\smallskipamount]
\def\+{\hskip 4pt+\hskip 4pt}
\left\{
\begin{array}{rcl}
a_{j}a_{j+1}\,X_{0}X_{1} & + & a_{\frac{N}{2}+j}a_{\frac{N}{2}+j+1}\,X_{\frac{N}{2}}X_{\frac{N}{2}+1} 
\\
& = & a_{0}a_{1}\  X_{j+1}X_{N-j}  \+  a_{\frac{N}{2}}a_{\frac{N}{2}+1} X_{\frac{N}{2}+j+1}X_{\frac{N}{2}-j},
\\
a_{0}a_{2j+1}\,X_{0}X_{1} & + & a_{\frac{N}{2}}a_{\frac{N}{2}+2j+1}\,X_{\frac{N}{2}}X_{\frac{N}{2}+1}
\\
& = & a_{j}a_{j+1}\,X_{j+1}X_{N-j} \+  a_{\frac{N}{2}+j}a_{\frac{N}{2}+j+1}\,X_{\frac{N}{2}+j+1}X_{\frac{N}{2}-j}.
\end{array}  \label{eq:eq-V1}
\right.
\end{gather}
}
Here, as before, the indices are considered modulo $N$.  As in the odd case, let $V=H^{0}(\P^{N-1},\mathcal{I}_{E_{\Theta_{\tau}}}(2))$, and $V_{k}=V\cap \<X_{i}X_{j}\mid i+j\equiv k \mod N\>$.  Equations \eqref{eq:eq-V0} (resp. \eqref{eq:eq-V1}) give quadratic forms in the space $V_{0}$ (resp. $V_{1}$). 
Because of the relation $a_j=a_{N-j}$,
it is easy to see that we may restrict ourselves to the case $0\le j\le \frac{N}2$.
Furthermore, letting $j=0\text{ or }\frac{N}2$ in \eqref{eq:eq-V0} gives trivial relations, and \eqref{eq:eq-V1} becomes also trivial for $j=\frac{N}2-1$.  Finally, the second equation of \eqref{eq:eq-V0} (resp. \eqref{eq:eq-V1})
is unchanged if we replace $j$ by $\frac{N}2-j$ (resp. $\frac{N}2-j-1$).  

\begin{theorem}\label{thm:eq-N-even}
Suppose $N$ is an even integer.  Let $E_{\Theta_{\tau}}$ be the image of $\Theta_{\tau}$ in Theorem~\ref{thm:theta-immersion}.  Let $V=H^{0}(\P^{N-1},\mathcal{I}_{E_{\Theta_{\tau}}}(2))$, and $V_{k}=V\cap \<X_{i}X_{j}\mid i+j\equiv k \mod N\>$, where $X_{i}$ $(i=0,\dots,N-1)$ are the coordinate functions of $\P^{N-1}$ in Definition~\ref{def:X_k-even}.  
Then, 
\begin{enumerate}
\item
The vector space $V_{2k}$ is isomorphic to $V_{0}$, and $V_{2k+1}$ is isomorphic to $V_{1}$ for all $k=0,\dots,\tfrac{N}{2}-1$ by the map $X_{i}X_{j}\mapsto X_{i+k}X_{j+k}$, where indices are taken as elements of $\Z/N\Z$.
\item
The quadratic forms
\begin{equation}\label{eq:qform-V0}
a_{j}^{2}\,X_{0}^{2}  +  a_{\frac{N}{2}+j}^{2}\,X_{\frac{N}{2}}^{2}  - 
a_{0}^{2}\,X_{j}X_{N-j}  -  a_{\frac{N}{2}}^{2}\,X_{\frac{N}{2}+j}X_{\frac{N}{2}-j}
\end{equation}
with $j=1,\ldots,\frac{N}{2}-1$ form a basis of the vector space $V_{0}$,
and the quadratic forms
\begin{equation}\label{eq:qform-V1}
\begin{aligned}
a_{j}a_{j+1}\,X_{0}X_{1}  &+  a_{\frac{N}{2}+j}a_{\frac{N}{2}+j+1}\,X_{\frac{N}{2}}X_{\frac{N}{2}+1} \\
& -  a_{0}a_{1}\  X_{j+1}X_{N-j}  -  a_{\frac{N}{2}}a_{\frac{N}{2}+1} X_{\frac{N}{2}+j+1}X_{\frac{N}{2}-j}
\end{aligned}
\end{equation}
with $j=1,\ldots,\tfrac{N}2-2$ form a basis of the vector space $V_{1}$.
\end{enumerate}
\end{theorem}

\begin{proof}  
(1) \ 
The isomorphism is nothing but the translation map $\tau_{S}^{k}$ in Proposition~\ref{prop:proj-coord-odd}.

(2) \ 
By \eqref{eq:eq-V0} and \eqref{eq:eq-V1}, we see that the quadratic forms in \eqref{eq:qform-V0} and \eqref{eq:qform-V1} belong to $V_{0}$ and $V_{1}$ respectively.
Looking at the indices of $X_{k}X_{l}$, the only possible 
dependency occurring between the quadratic forms in \eqref{eq:qform-V0} are between the ones $j=j_0$ and $j=\frac{N}2-j_0$. But then the 
determinant of the two by two matrix of coefficients of $X_{j_0}X_{-j_0}$
and $X_{\frac{N}2+j_0}X_{\frac{N}2-j_0}$ is $a_0^4-a_{\frac{N}2}^4$, and it is non-zero, as is seen by looking at its Fourier
series.  So, these $\frac{N}{2}-1$ forms are linearly independent.
The same argument applies for the quadratic forms in \eqref{eq:qform-V1}, and the $\frac{N}{2}-2$ forms are linearly independent.  Since $\dim V=N(N-3)/2$, we have $\dim V_{0}+ \dim V_{1}= N-3$ by (1).  This implies that the $N-3$ independent forms in \eqref{eq:qform-V0} and \eqref{eq:qform-V1}  form a basis of $V_{0}$ and $V_{1}$ respectively.
\end{proof}

\begin{remark}
As before, by specializing $X_i$ at $z=0$ in \eqref{eq:deg2-N-odd}, we obtain quartic relations in $a_{j}$'s.  These quartic equations are satisfied by the image of the morphism $\alpha:Y^{(N)}(2N)\to \P^{N-1}$ in \S\ref{subsec:modular_curves}.  We will study these equations in the examples in the following sections.
\end{remark}

\section{Level 4}\label{sec:level-4}

In this section we consider the case $N=4$.  
As is well known, the congruence subgroups $\Gamma(4)$ and $\Gamma(8)$ are of genus~$0$ and~$5$, respectively, and according to the database \cite{Congr-Data}, $\Gamma^{(4)}(8)$ is of genus~$3$. (In the notation of \cite{Congr-Data}, $\Gamma^{(4)}(8)$ is denoted by ``$8B^{3}$''.)
Theorem~\ref{thm:eq-N-even} shows that 
$V=H^{0}(\P^{3},\mathcal{I}_{E_{\Theta_{\tau}}}(2))=\bigoplus_{i=0}^{3}V_{i}$ is given by
\begin{align*}
&V_{0} = \< a_{1}^{2}(X_{0}^{2} + X_{2}^{2}) 
- (a_{0}^{2}+a_{2}^{2})X_{1}X_{3}\>,
\\
&V_{1}=V_{3} = \{0\}.
\\
&V_{2} = \< a_{1}^{2}(X_{1}^{2} + X_{3}^{2})
- (a_{0}^{2}+a_{2}^{2})X_{0}X_{2}\>,
\end{align*}
Thus, the following two equations define the universal elliptic curve $\E^{(4)}(8)$ over $X^{(4)}(8)$.
\begin{equation}\label{eq:E4_8}
\renewcommand{\arraystretch}{1.2}
\setlength{\arraycolsep}{2pt}
\E^{(4)}(8) : \ 
\left\{
\begin{array}{*3c}
a_{1}^{2}(X_{0}^{2} + X_{2}^{2}) &= & (a_{0}^{2}+a_{2}^{2}) X_{1}X_{3},
\\
(a_{0}^{2}+a_{2}^{2}) X_{0}X_{2} &= & a_{1}^{2}(X_{1}^{2} + X_{3}^{2}).
\end{array}\right.
\end{equation}
Letting $z=0$, in other words, replacing $X_{i}$ by $a_{i}$, we obtain two equations satisfied by $a_{i}$'s.  However, considering the fact $a_{1}=a_{3}$, we obtain only one nontrivial relation
\begin{equation}\label{eq:X4_8}
 a_{0}a_{2}(a_{0}^{2} + a_{2}^{2}) = 2a_{1}^{4}.
\end{equation}
This equation defines a nonsingular curve in $\P^{2}$, and it is a curve of genus~$3$.  On the other hand, as we mentioned above, the genus of $X^{(4)}(8)$ equals~$3$.  Thus, the plane curve defined by~\eqref{eq:X4_8} is nothing but $X^{(4)}(8)$.

\begin{proposition}\label{prop:E4_8}
The curve $\E^{(4)}(8)$ is isomorphic to
\begin{equation}\label{eq:Weier_E4_8}
Y^{2}=X\bigl(X-(a_{0}-a_{2})^4\bigr)\bigl(X-(a_{0}+a_{2})^4\bigr).
\end{equation}
The points in $X^{(4)}(8)$ at which the fiber of $\E^{(4)}(8)$ degenerates into a N\'eron polygon of four sides are
\[
(a_{0}:a_{1}:a_{2})
= (0:0:1), (1:0:0), (\pm\zeta_{8}^{2}:0:1), (1:\zeta_{8}^{k}:1) \ (k=0,\dots,7),
\]
where $\zeta_{8}$ is a primitive eighth root of unity.
\end{proposition}

\begin{proof}
From two quadrics \eqref{eq:E4_8}, eliminate $X_{1}$ to obtain a quartic 
equation.  Then, using the rational point $(X_{0}:X_{1}:X_{2}:X_{3})=(a_{0}:a_{1}:a_{2}:a_{1})$, we obtain the Weierstrass equation \eqref{eq:Weier_E4_8} after some simplifications.  The change of coordinates is given by
\begin{equation}\label{eq:chang_E4_8}
\left\{\begin{aligned}
X &= 
(a_{0}^{2}-a_{2}^{2})^2\frac{
(a_{1}^{2} X_{0} X_{2}+a_{0}a_{2} X_{1}X_{3}) }
{(a_{1}^{2}X_{0} X_{2}-a_{0}a_{2} X_{1} X_{3}) }, 
\\
Y &= 
4a_{1}^{2}(a_{0}^{2}-a_{2}^{2})^2\frac{
(X_{1}+X_{3})(X_{0}+X_{2})(a_{0}a_{2}X_{0}X_{2}-a_{1}^{2} X_{1} X_{3})}
{(X_{1}-X_{3})(X_{0}-X_{2})(a_{1}^{2} X_{0} X_{2}-a_{0}a_{2} X_{1} X_{3})}.
\end{aligned}\right.
\end{equation}
The locus of degenerate fibers are where the right hand side of \eqref{eq:Weier_E4_8} has a multiple root.
\end{proof}

The action of $SL_{2}(\Z)/\Gamma^{(4)}(8)$ on $\E^{(4)}(8)$ and $X^{(4)}(8)$ are as follows:
\begin{alignat*}{2}
\rho\tsltwo(0,-1;1,0)
&=\left[\begin{array}{rcrc}
1 & 1 & 1 & 1 \\
1 &\zeta_{8}^{2} & -1 & -\zeta_{8}^{2} \\
1 & -1 & 1 & -1 \\
1 &-\zeta_{8}^{2} & -1 & \zeta_{8}^{2} 
\end{array}\right],
\quad &
\rho\tsltwo(1,1;0,1)
&=\left[\begin{array}{rcrc}
1 & 0 & 0 & 0 \\
0 &\zeta_{8}^{3} & 0 &0 \\
0 & 0 & -1 & 0 \\
0 & 0 &  0 & \zeta_{8}^{3} 
\end{array}\right],
\\
\bar\rho\tsltwo(0,-1;1,0)
&=\left[\begin{array}{rrr}
1 & 2 & 1 \\
1 & 0 & -1\\
1 & -2 & 1 \\
\end{array}\right],
\quad &
\bar\rho\tsltwo(1,1;0,1)
&=\left[\begin{array}{rcrc}
1 & 0 & 0\\
0 &\zeta_{8}^{3} & 0 \\
0 & 0 & -1 \\
\end{array}\right].
\end{alignat*}
In particular, we have
\begin{align*}
\rho\tsltwo(1,4;0,1)&:
(X_{0}:X_{1}:X_{2}:X_{3}) \mapsto  
(X_{0}:-X_{1}:X_{2}:-X_{3}),
\\
\rho\tsltwo(1,0;4,1)&:
(X_{0}:X_{1}:X_{2}:X_{3}) \mapsto  
(X_{2}:X_{3}:X_{0}:X_{1}),
\\
\bar\rho\tsltwo(1,4;0,1)&:
(a_{0}:a_{1}:a_{2}) \mapsto  
(a_{0}:-a_{1}:a_{2}),
\\
\bar\rho\tsltwo(1,0;4,1)&:
(a_{0}:a_{1}:a_{2}) \mapsto  
(a_{2}:a_{1}:a_{0}).
\end{align*}
Let $G=\Gamma(4)/\Gamma^{(4)}(8)=\<\tsltwo(1,4;0,1)
\Gamma^{(4)}(8),\tsltwo(1,0;4,1)
\Gamma^{(4)}(8)\>\simeq \Z/2\Z\times \Z/2\Z$.  In order to obtain a model of the universal elliptic curve $\E(4)\to X(4)$, 
we take the quotient of $\E^{(4)}(8)\to X^{(4)}(8)$ by the action of~$G$.  
To do so, we first dehomogenize the equations \eqref{eq:E4_8} and \eqref{eq:X4_8} by letting 
\[
x_{0}=\frac{X_{0}}{X_{3}}, \quad
x_{1}=\frac{X_{1}}{X_{3}}, \quad
x_{2}=\frac{X_{2}}{X_{3}}, \quad
\alpha_{0}=\frac{a_{0}}{a_{1}}, \quad
\alpha_{2}=\frac{a_{2}}{a_{1}}, 
\]
and consider the function field $K(\alpha_{0},\alpha_{2},x_{0},x_{1},x_{2})$.  Then, these variables satisfy the following relations:
\begin{gather}\label{eq:X4_8-a}
\alpha_{0}\alpha_{2}(\alpha_{0}^{2} + \alpha_{2}^{2}) = 2, \\
\label{eq:E4_8-a}
\left\{
\begin{aligned}
& \alpha_{0}\alpha_{2}(x_{0}^{2} + x_{2}^{2}) = 2 x_{1},\\
& 2 x_{0}x_{2} =  \alpha_{0}\alpha_{2}(x_{1}^{2} + 1).
\end{aligned}\right.
\end{gather}
Let $\sigma_{1}=\tsltwo(1,4;0,1)\Gamma^{(4)}(8)$ and $\sigma_{2}=\tsltwo(1,0;4,1)\Gamma^{(4)}(8)$ be the generators of $G$.  Then, they act as follows:
\begin{align*}
\sigma_{1}&:
(\alpha_{0},\alpha_{2},x_{0},x_{1},x_{2}) \mapsto  
(-\alpha_{0},-\alpha_{2},-x_{0},x_{1},-x_{2}),
\\
\sigma_{2}&:
(\alpha_{0},\alpha_{2},x_{0},x_{1},x_{2}) \mapsto  
\left(\alpha_{2},\alpha_{0},\frac{x_{2}}{x_{1}},\frac{1}{x_{1}},\frac{x_{0}}{x_{1}}\right).
\end{align*}
It is easy to see that $K(\alpha_{0},\alpha_{2})^{G}=K(\alpha_{0}\alpha_{2},\alpha_{0}^{2}+\alpha_{2}^{2})$.  Considering \eqref{eq:X4_8-a}, we see that $K(\alpha_{0},\alpha_{2})^{G}=K(\lambda)$ with $\lambda=\alpha_{0}\alpha_{2}$.  
This shows that $X(4)\simeq \P^{1}_{\lambda}$.  
Further calculations show that the fixed field $K(\alpha_{0},\alpha_{2},x_{0},x_{1},x_{2})^{G}$ is generated by 
\[
\lambda=\alpha_{0}\alpha_{2}, \quad
\xi_{0}=\frac{x_{0}x_{2}}{x_{1}}, \quad
\xi_{1}=x_{1}+\frac{1}{x_{1}}, \quad
\xi_{2}=\frac{(x_{1}+1)(x_{0}+x_{2})}{(x_{1}-1)(x_{0}-x_{2})}.
\]
Now, \eqref{eq:chang_E4_8} can be written in terms of $\lambda$, $\xi_{0}$ and $\xi_{1}$:
\[
X=\frac{4(1-\lambda^4)}{\lambda^{2}}
\frac{(\xi_{0}+\lambda)}{(\xi_{0}-\lambda)}, \quad
Y=\frac{16(1-\lambda^4)}{\lambda^{2}}
\frac{\xi_{2}(1-\lambda \xi_{0})}
{(\xi_{0}-\lambda)}.
\]
If we let $X'=\lambda^{2}X/4$ and $Y'=\lambda^{2}Y/8$,
then equation \eqref{eq:Weier_E4_8} becomes
\[
\E(4): Y'^{2} = X'\bigl(X'-(\lambda^2-1)^{2}\bigr)
\bigl(X'-(\lambda^2+1)^{2}\bigr).
\]
The representaion $\varrho:SL_{2}(\Z/4\Z) \to \Aut(\E(4))$ is given by
\begin{align*}
\varrho\left(\begin{smallmatrix} 0 & -1 \\ 1 & 0\end{smallmatrix}\right): 
(\lambda,X',Y')&\mapsto 
\left(\frac{-\lambda+1}{\lambda+1}, 
-\frac{4(X'+(\lambda^2+1)^2)}{(\lambda+1)^4},
-\frac{8\zeta_{8}^{2}Y'}{(\lambda+1)^6}\right),\\
\varrho\left(\begin{smallmatrix} 1 & 1 \\ 0 & 1\end{smallmatrix}\right):
(\lambda,X',Y')&\mapsto ( -\zeta_{8}^{2}\lambda, X',Y').
\end{align*}

Note that we have
\begin{align*}
\lambda(\tau) &= \frac{a_{0}(\tau)a_{2}(\tau)}{a_{1}(\tau)^{2}} =
2\left(\frac{\eta(\tau)\eta(4\tau)^2}{\eta(2\tau)^3}\right)^2\\
&=2q^{\frac14} -4q^{\frac54}+10q^{\frac94}-20q^{\frac{13}4}+36q^{\frac{17}4}+\cdots  
\quad (q=e^{2\pi i \tau}),
\end{align*}
where $\eta(\tau)$ is the Dedekind eta function.

\section{Level 6}\label{sec:level-6}
Next, we consider the case $N=6$.  The congruence subgroups $\Gamma(6)$ and $\Gamma(12)$ are of genus~$1$ and~$25$, respectively, and according to the database \cite{Congr-Data}, $\Gamma^{(6)}(12)$ is denoted by ``$12B^{13}$'', and of genus~$13$.  The basis of $V_{k}$ obtained by Theorem~\ref{thm:eq-N-even} is as follows:
\allowdisplaybreaks
\begin{equation}\label{eq:level-6}
\begin{gathered}
V_{0} = \< a_{1}^{2}X_{0}^{2} + a_{2}^{2}X_{3}^{2}
- a_{0}^{2}X_{1}X_{5} - a_{3}^{2}X_{2}X_{4}, \ 
a_{2}^{2}X_{0}^{2} + a_{1}^{2}X_{3}^{2}
- a_{3}^{2}X_{1}X_{5} - a_{0}^{2}X_{2}X_{4}\>,
\\
V_{2} = \< a_{1}^{2}X_{1}^{2} + a_{2}^{2}X_{4}^{2}
- a_{0}^{2}X_{2}X_{0} - a_{3}^{2}X_{3}X_{5}, \ 
a_{2}^{2}X_{0}^{2} + a_{1}^{2}X_{3}^{2}
- a_{3}^{2}X_{2}X_{0} - a_{0}^{2}X_{3}X_{5}\>,
\\
V_{4} = \< a_{1}^{2}X_{2}^{2} + a_{2}^{2}X_{5}^{2}
- a_{0}^{2}X_{3}X_{1} - a_{3}^{2}X_{2}X_{4}, \ 
a_{2}^{2}X_{0}^{2} + a_{1}^{2}X_{5}^{2}
- a_{3}^{2}X_{3}X_{1} - a_{0}^{2}X_{4}X_{0}\>,
\\
V_{1} = \< a_{1}a_{2} (X_{0}X_{1} + X_{3}X_{4})
- (a_{0}a_{1}+a_{2}a_{3} )X_{2}X_{5}\>,
\\
V_{3} = \< a_{1}a_{2} (X_{1}X_{2} + X_{4}X_{5})
- (a_{0}a_{1}+a_{2}a_{3} )X_{3}X_{0}\>, 
\\
V_{5} = \< a_{1}a_{2} (X_{2}X_{3} + X_{5}X_{0})
- (a_{0}a_{1}+a_{2}a_{3} )X_{4}X_{1}\>.
\end{gathered}
\end{equation}
By replacing $X_{i}$ by $a_{i}$, we obtain quartic relations among 
$a_i$'s.  Noting that $a_5=a_1$ and $a_4=a_2$, we obtain only two relations:
\begin{equation}\label{eq:level-6-homogeneous}
\left\{
\begin{gathered}
a_{1}^{4}+a_{2}^{4}=a_{0}^{3}a_{2}+a_{1}a_{3}^{3}, \\
a_{0}a_{3} \left( a_{0}a_{1}+a_{2}a_{3} \right) 
=2\,a_{1}^{2}a_{2}^{2}.
\end{gathered}
\right.
\end{equation}
The curve in $\P^{3}$ defined by \eqref{eq:level-6-homogeneous} turns out to be reducible.  Computations using Groebner basis reveal that there are five irreducible components; four of them are lines
\begin{gather*}
a_{1}=a_{2}=0, \quad
a_{0}-a_{2}=a_{1}-a_{3}=0, \\
a_{0}-\omega a_{2}=\omega a_{1}-a_{3}=0, \quad
a_{0}-\omega^{2} a_{2}=\omega^{2} a_{1}-a_{3}= 0, 
\end{gather*}
where $\omega$ is a primitive third root of unity, and the other component is an irreducible curve of genus~$13$.  Since the congruence subgroup 
$\Gamma^{(6)}(12)$ is of genus~$13$, the last 
irreducible component given by
\[
a_{2}(a_{1}^4 - a_{2}^4)a_{0}^6 - (a_{1}^8 - a_{2}^8)a_{0}^3 + 8a_{2}^3a_{1}^4(a_{0}^3a_{2} - a_{1}^4)=0
\]
is the modular curve $X^{(6)}(12)$ associated with the group $\Gamma^{(6)}(12)$.

The action of $SL_{2}(\Z)/\Gamma^{(6)}(12)$ on $X^{(4)}(8)$ is as follows:
\begin{gather*}
\bar\rho\tsltwo(0,-1;1,0)
=\left[\begin{array}{rcrc}
1 & 2 & 2 & 1 \\
1 & 1 & -1 & -1\\
1 & -1 & -1 & 1\\
1 & -2 & 2 & -1 
\end{array}\right],
\quad
\bar\rho\tsltwo(1,1;0,1)
=\left[\begin{array}{*4c}
1 & 0 & 0 & 0\\
0 &\zeta_{12}^{5} & 0 & 0\\
0 & 0 &\zeta_{12}^{8} & 0 \\
0 & 0 & 0 &\zeta_{12}^{9} 
\end{array}\right].
\end{gather*}
In particular, we have
\begin{align*}
\bar\rho\tsltwo(1,6;0,1)&:
(a_{0}:a_{1}:a_{2}:a_{3}) \mapsto  
(a_{0}:-a_{1}:a_{2}:-a_{3}),
\\
\bar\rho\tsltwo(1,0;6,1)&:
(a_{0}:a_{1}:a_{2}:a_{3}) \mapsto  
(a_{3}:a_{2}:a_{1}:a_{0}).
\end{align*}

Let us find a model of $X(6)$ by taking the quotient of $X^{(6)}(12)$ by 
$\Gamma(6)/\Gamma^{(6)}(12)=\<\bar\rho\tsltwo(1,6;0,1),\bar\rho\tsltwo(1,0;6,1)\>\simeq \Z/2\Z\times \Z/2\Z$. 
Dehomogenizing the coordinates by
$\alpha_{1}={a_{1}}/{a_{0}}$, $\alpha_{2}={a_{2}}/{a_{0}}$,
$\alpha_{3}={a_{3}}/{a_{0}}$,
the equation \eqref{eq:level-6-homogeneous} becomes
\begin{equation}\label{eq:level-6-inhomogeneous}\left\{
\begin{gathered}
\alpha_{1}^{4}+\alpha_{2}^{4}=\alpha_{2}+\alpha_{1}\alpha_{3}^{3}, \\
\alpha_{3}(\alpha_{1}+\alpha_{2}\alpha_{3})=2\alpha_{1}^{2}\alpha_{2}^{2}.
\end{gathered}
\right.
\end{equation}
Let $\sigma_{1}$ and $\sigma_{2}$ denote the automorphisms induced on the function field $k(\alpha_{1},\alpha_{2},\alpha_{3})$ by $\bar\rho\tsltwo(1,6;0,1)$ and $\bar\rho\tsltwo(1,0;6,1)$, respectively.  We have
\[
\sigma_{1} : (\alpha_{1},\alpha_{2},\alpha_{3}) \mapsto (-\alpha_{1},\alpha_{2},-\alpha_{3}), \quad
\sigma_{2} : (\alpha_{1},\alpha_{2},\alpha_{3}) \mapsto 
\Bigl(\frac{\alpha_{2}}{\alpha_{3}},\frac{\alpha_{1}}{\alpha_{3}},\frac{1}{\alpha_{3}}\Bigr).
\]
The fixed subfields by $\sigma_{1}$ and $\sigma_{2}$ are 
\begin{align*}
&k(\alpha_{1},\alpha_{2},\alpha_{3})^{\sigma_{1}}
=k\Bigl(\alpha_{1}\alpha_{3},\frac{\alpha_{1}}{\alpha_{3}},\alpha_{2}\Bigr), \\
&k(\alpha_{1},\alpha_{2},\alpha_{3})^{\sigma_{2}}
=\Bigl(\alpha_{1}+\frac{\alpha_{2}}{\alpha_{3}},\alpha_{2}+\frac{\alpha_{1}}{\alpha_{3}},\alpha_{3}+\frac{1}{\alpha_{3}}\Bigr).
\end{align*}
Finally, define
\begin{equation}\label{eq:level8-invariant}
\beta_{1}=\frac{\alpha_{1}\alpha_{2}}{\alpha_{3}}, \quad
\beta_{2}=\alpha_{1}\alpha_{3}+\frac{\alpha_{2}}{\alpha_{3}^{2}}.
\end{equation}
Clearly, $\beta_{1}$ and $\beta_{2}$ are fixed by $G=\<\sigma_{1},\sigma_{2}\>$.  It is easy to show that the fixed field by the group $G$ is given by
\[
k(\alpha_{1},\alpha_{2},\alpha_{3})^{G} = k(\beta_{1},\beta_{2}).
\]
Eliminating $\alpha_{1},\alpha_{2},\alpha_{3}$ from \eqref{eq:level-6-inhomogeneous} and \eqref{eq:level8-invariant}  by some calculations based on Groebner basis, we find that if we let
\[
X = 2 \beta_{1}, 
\quad
Y = \frac{2 \beta_{1}^2(4 \beta_{1}^3 - 1)}{\beta_{2}},
\]
they satisfy
\[
Y^{2} = X^{3} + 1,
\]
which is well-known model of the modular curve~$X(6)$.
Furthermore, we have 
\[
\beta_{1} = \frac{X}{2},\quad
\beta_{2} = \frac{X^2(X^3 - 2)}{4Y},
\]
and thus $k(\beta_{1},\beta_{2})=k(X,Y)$. 

On the curve $Y^{2}=X^{3}+1$, the actions of $\bar\rho\sltwo(0,1;-1,0)$ and $\bar\rho\sltwo(1,1;0,1)$ are given by
\begin{align*}
\bar\rho\tsltwo(0,1;-1,0)
&:(X,Y) \mapsto (2,-3) - (X,Y),
\\
\bar\rho\tsltwo(1,1;0,1)
&:(X,Y) \mapsto [-\omega](X,Y)=(\omega X,-Y),
\end{align*}
where the operation ``$-$'' in the first map is the group operation of $Y^{2}=X^{3}+1$, and the map $[\omega]$ is the complex multiplication of $Y^{2}=X^{3}+1$.
In terms of $a_{i}(\tau)=\theta_{i}^{(N)}(0,\tau)$, $X$ and $Y$ are expressed as follows:
\[
X = \frac{2a_{1}(\tau)a_{2}(\tau)}{a_{0}(\tau)a_{3}(\tau)}, \quad
Y = \frac{a_{0}(\tau)^{2}a_{1}(\tau)^{2}
-a_{2}(\tau)^{2}a_{3}(\tau)^{2}}
{a_{0}(\tau)^{2}a_{2}(\tau)^{2}
-a_{1}(\tau)^{2}a_{3}(\tau)^{2}}.
\]
Incidentally, $X$ and $Y$, which are modular functions on $\Gamma(6)$, can also be written by using the 
Dedekind eta function as follows:
\begin{align*} X&=\frac{\eta(2\tau)\eta(3\tau)^3}{\eta(\tau)\eta(6\tau)^3}=q^{-\frac13} + q^{\frac{2}{3}} + q^{\frac{5}{3}} 
- q^{\frac{8}{3}} - q^{\frac{11}{3}} + q^{\frac{17}{3}} + 2\, q^{\frac{20}{3}} -\cdots,
\\
Y&=\frac{\eta(2\tau)^4\eta(3\tau)^2}{\eta(\tau)^2\eta(6\tau)^4}=q^{-\frac12} + 2 q^{\frac{1}{2}} + q^{\frac{3}{2}} 
- 2 q^{\frac{7}{2}} - 2 q^{\frac{9}{2}} +  2 q^{\frac{11}{2}} + 4 q^{\frac{13}{2}} + \cdots,
\end{align*}
where $q=e^{2\pi i \tau}$. The function $Y$ has a simpler expression in terms of $a_i(\tau)=\theta_i^{(6)}(0,\tau)$ as  
\[ Y=\frac{a_0(\frac{\tau}{3})a_3(\frac{\tau}{3})}{a_0(\tau)a_3(\tau)}. \]

At the end of this section, we note the connection to the ``Hesse cubic'', that is the elliptic normal
curve of degree 3.  Let
\begin{align*}  3\mu&=X^2-\frac{2}{X} =\frac{Y^2-3}{X}\\
&=q^{-\frac23}+5q^{\frac43}-7q^{\frac{10}3}+3q^{\frac{16}3}+15q^{\frac{22}3}-32q^{\frac{28}3} 
+9q^{\frac{34}3}+\cdots.   
\end{align*}
We can check that $\mu(\tau/2)$ is a modular function for the group $\Gamma(3)$, 
and we have the relation 
\[ X_0^3+X_2^3+X_4^3=3\mu(\tau) X_0 X_2 X_4. \]
Actually, our theta functions $\theta_k^{(3)}(z,\tau)\ (k=0,1,2)$ for $N=3$
can be obtained from those for $N=6$ ($k=0,2,4$) by changing the variables $z\to z/2-1/4, \,\tau\to\tau/2$,
and thus the above equation gives the Hesse cubic family of elliptic curves with level 3 structure.
We refer the reader to \cite{KKNT} for the derivation of the Hesse cubic in the same line of the current paper.

\section{Level 8}\label{sec:level-8}

In the case $N=8$, the basis of $V_{0}$ and $V_{1}$ obtained by Theorem~\ref{thm:eq-N-even} is as follows:
\begin{equation}\label{eq:level-8}
\renewcommand{\arraystretch}{1.0}
\setlength{\arraycolsep}{2pt}
\begin{aligned}
&\begin{array}{r*{10}c}
V_{0}=\< &a_{1}^{2}X_{0}^{2} &+& a_{3}^{2}X_{4}^{2} 
&-& a_{0}^{2}X_{1}X_{7}& & &-& a_{4}^{2}X_{3}X_{5}, 
\\
&a_{2}^{2}X_{0}^{2} &+& a_{2}^{2}X_{4}^{2} 
& & &-& (a_{0}^{2}+a_{4}^{2})X_{2}X_{6}, &
\\
&a_{3}^{2}X_{0}^{2} &+& a_{1}^{2}X_{4}^{2} 
&-& a_{4}^{2}X_{1}X_{7} & & &-& a_{0}^{2}X_{3}X_{5} &\>,
\end{array}
\\
&\begin{array}{r*8c}
V_{1}=\< &a_{1}a_{2}X_{0}X_{1} &-& a_{0}a_{1}X_{2}X_{7} 
&-& a_{3}a_{4}X_{3}X_{6} &+& a_{2}a_{3}X_{4}X_{5},
\\
&a_{2}a_{3}X_{0}X_{1} &-& a_{3}a_{4}X_{2}X_{7} 
&-& a_{0}a_{1}X_{3}X_{6} &+& a_{1}a_{2}X_{4}X_{5} &\>.
\end{array}
\end{aligned}
\end{equation}
The bases of $V_{2k}$ and $V_{2k+1}$ are obtained by replacing $X_{i}$ by $X_{i+k}$ in $V_{0}$ and $V_{1}$ respectively.
By letting $z=j\tau/N$, $j=0,1,\dots,$ in \eqref{eq:level-8}, we obtain the following relations among~$a_{i}$:
\begin{equation}\label{eq:level-8-modular}
\left\{
\begin{aligned}
&a_{0}a_{4}(a_{0}^{2}+a_{4}^{2}) = 2a_{2}^{4},
\\
&a_{0}a_{4}(a_{1}^{2}+a_{3}^{2}) = 2a_{1}a_{3}a_{2}^{2},
\\
&a_{0}a_{2}a_{4}(a_{0}+a_{4})=2a_{1}^{2}a_{3}^{2},
\\
&a_{2}^{3}(a_{0}+a_{4}) = a_{1}a_{3}(a_{1}^{2}+a_{3}^{2}),
\\
&a_{1}a_{3}(a_{0}^{2}+a_{4}^{2}) = a_{2}^{2}(a_{1}^{2}+a_{3}^{2}),
\\
&a_{2}(a_{0}^{3} + a_{4}^{3}) = a_{1}^{4} + a_{3}^{4}.
\end{aligned}
\right.
\end{equation}
The curve in $\P^{4}$ defined by \eqref{eq:level-8-modular} turns out to be  an irreducible curve of genus~$41$, and this is a model of $X^{(8)}(16)$.

Dehomogenize \eqref{eq:level-8-modular} by letting
$\alpha_{0}={a_{0}}/{a_{2}}$, $\alpha_{1}={a_{1}}/{a_{2}}$, $\alpha_{3}={a_{3}}/{a_{2}}$, $\alpha_{4}={a_{4}}/{a_{2}}$.
By some calculations of Groebner basis, it turns out that in the function field $k(\alpha_{0},\alpha_{1},\alpha_{3},\alpha_{4})$, the following three equations are enough to generate the ideal generated by the above six equations:
\begin{equation}\label{eq:level8-rel}
\alpha_{0}\alpha_{4}(\alpha_{0}^{2}+\alpha_{4}^{2}) = 2,\quad 
\alpha_{0}+\alpha_{4} = \alpha_{1}\alpha_{3}(\alpha_{1}^{2}+\alpha_{3}^{2}),\quad 
\alpha_{1}\alpha_{3}(\alpha_{0}^{2}+\alpha_{4}^{2}) = \alpha_{1}^{2}+\alpha_{3}^{2}.
\end{equation}

Let $\sigma_{1}$ and $\sigma_{2}$ denote the automorphisms of $k(\alpha_{0},\alpha_{1},\alpha_{3},\alpha_{4})$ induced by $\bar\rho\tsltwo(1,8;0,1)$ and $\bar\rho\tsltwo(1,0;8,1)$, respectively.  We have
\begin{align*}
& \sigma_{1} : (\alpha_{0},\alpha_{1},\alpha_{3},\alpha_{4}) \mapsto 
(\alpha_{0},-\alpha_{1},-\alpha_{3},\alpha_{4}), \\
&\sigma_{2} : (\alpha_{0},\alpha_{1},\alpha_{3},\alpha_{4}) \mapsto 
(\alpha_{4},\alpha_{3},\alpha_{1},\alpha_{0}).
\end{align*}
It is easy to see that $\alpha$ and $\beta$ commute, and thus the group of automorphism $\<\alpha,\beta\>$ induced by $\alpha$ and $\beta$ is isomorphic to $(\Z/2\Z)^{2}$.  The fixed subfields by $\alpha$ and $\beta$ are 
\begin{align*}
k(\alpha_{0},\alpha_{1},\alpha_{3},\alpha_{4})^{\sigma_{1}}
&=k\Bigl(\alpha_{0},\alpha_{1}\alpha_{3},\frac{\alpha_{1}}{\alpha_{3}},\alpha_{4}\Bigr), \\
k(\alpha_{0},\alpha_{1},\alpha_{3},\alpha_{4})^{\sigma_{2}}
&=k\left(\alpha_{0}+\alpha_{4},\alpha_{1}+\alpha_{3}, (\alpha_{1}-\alpha_{3})(\alpha_{0}-\alpha_{4}),\frac{\alpha_{1}-\alpha_{3}}{\alpha_{0}-\alpha_{4}}\right).
\end{align*}
Finally, define
\[
\beta_{0}=\alpha_{0}+\alpha_{4},\quad
\beta_{1}=\alpha_{1}\alpha_{3},\quad
\beta_{3}=\frac{\alpha_{1}}{\alpha_{3}}+\frac{\alpha_{3}}{\alpha_{1}},\quad
\beta_{4}=\frac{\alpha_{1}\alpha_{3}(\alpha_{0}-\alpha_{4})}{(\alpha_{1}+\alpha_{3})(\alpha_{1}-\alpha_{3})}.
\]
Then, $\beta_{0}$, $\beta_{1}$, $\beta_{3}$, and $\beta_{4}$ are all fixed by $G=\<\sigma_{1},\sigma_{2}\>$.  It is easy to show that the fixed field by the group $G$ is given by

\[
k(\alpha_{0},\alpha_{1},\alpha_{3},\alpha_{4})^{G}
=k(\beta_{0},\beta_{1},\beta_{3},\beta_{4}).
\]
Further calculations show that the relations \eqref{eq:level8-rel} translates to the relations
\[
\beta_{0}=\beta_{1}\beta_{3}, \quad \beta_{3}=\frac{1}{\beta_{4}^{2}}, \quad
\beta_{1}^{4}=4\beta_{4}^{6}+\beta_{4}^{2}.
\]
This implies that
\[
k(\alpha_{0},\alpha_{1},\alpha_{3},\alpha_{4})^{G}
=k(\beta_{1},\beta_{4}), \quad \text{with } \beta_{1}^{4}=4\beta_{4}^{6}+\beta_{4}^{2}
\]
The genus of the curve defined by $\beta_{1}^{4}=4\beta_{4}^{6}+\beta_{4}^{2}$ is~$5$, which coincides with the genus of $X(8)$.  Thus, we conclude that an affine model of $X(8)$ is given by 
\[
X(8): \beta_{4}^{2}(1+4\beta_{4}^{2})=\beta_{1}^{4}.
\]
Recall that an equation of $X^{(4)}(8)$ is given by $ a^{(4)}_{0}a^{(4)}_{2}\bigl((a^{(4)}_{0})^{2} + (a^{(4)}_{2})^{2}\bigr) = 2(a^{(4)}_{1})^{4}$ (see \eqref{eq:X4_8}).
We have a two-to-one map 
\[
X(8)\to X^{(4)}(8); \quad
(\beta_{1},\beta_{4})\mapsto
\bigl(a^{(4)}_{0}:a^{(4)}_{2}:a^{(4)}_{4}\bigr)
=(1:\beta_{1}:2\beta_{4}^{2}).
\]

The modular functions $\beta_1$ and $\beta_4$ on $\Gamma(8)$ can be written in terms of the Dedekind eta function as follows:
\begin{align*} 
\beta_1&=\frac{\eta(2\tau)^4\eta(8\tau)^2}{\eta(\tau)\eta(4\tau)^5}=q^{\frac{1}{8}} + q^{\frac{9}{8}} - 2 q^{\frac{17}{8}} - q^{\frac{25}{8}}
+ 4 q^{\frac{33}{8}} + 2 q^{\frac{41}{8}} -  7 q^{\frac{49}{8}}  -\cdots,,
\\
\beta_4&=-\frac{\eta(2\tau)\eta(8\tau)^2}{\eta(4\tau)^3}=q^{\frac{1}{4}} - q^{\frac{9}{4}} + 2 q^{\frac{17}{4}} - 3 q^{\frac{25}{4}} + 4 q^{\frac{33}{4}} - 
 6 q^{\frac{41}{4}} + 9 q^{\frac{49}{4}}- \cdots.
\end{align*}
They also have neat expressions in terms of our theta series:
\[ 
\beta_1=\frac{a_2(\frac{\tau}{2})a_2(2\tau)}{a_2(\tau)^2},\quad \beta_4=\frac{a_2(2\tau)}{a_2(\tau)}. 
\]

\appendix

\section{Hurwitz's immersion}\label{appendix}

Here, we describe the connections and differences between Hurwitz's $\sigma$-functions and our theta
functions in detail.

Hurwitz \cite{Hurwitz} uses the Weierstrass $\sigma$-function to construct an immersion of an elliptic curve into projective space. The $\sigma$-function is defined with respect to a fundamental pair of periods $\omega _{1},\omega _{2}\in \C$. Replacing $\omega_{2}$ by $-\omega_{2}$ if necessary, we may assume $\Im\omega_{2}/\omega_{1}>0$.  So, if we define $\tau=\omega_{2}/\omega_{1}$, $\tau$ is a point in the upper half plane $\H$.

Let $\Lambda$ be the lattice $\Lambda =\{m\omega _{1}+n\omega _{2}\,\,|\,\,m,n\in \Z \}$, and $\Lambda^{*}$ denotes the set of nonzero elements of $\Lambda$. The Weierstrass sigma function $\sigma(u)$ associated with $\Lambda \subset \C$ is defined by
\[
\sigma(u)=u\prod_{w\in\Lambda^{*}}
\Bigl(1-\tfrac{u}{w}\Bigr) e^{\frac{u}{w}+\frac{1}{2}(\frac{u}{w})^2},
\]
The function $\sigma$ is a quasiperiodic holomorphic function having a simple zero at each of the points in $\Lambda$.  Here, ``quasiperiodic'' means that $\sigma(u)$ satisfies the identities $\sigma(u+\omega_{k})
=\exp\bigl({\eta_{k}\bigl(u+\tfrac{\omega_{k}}{2}\bigr)+\pi i}\bigr)\sigma(u)$ \ ($k=1,2$), where $\eta_{1}$, $\eta_{2}$ satisfies the identity
$\eta_{1}\omega_{2} - \eta_{2}\omega_{1}=2\pi i$.
We sometimes write $\sigma(u\mid \omega _{1},\omega _{2})$ instead of $\sigma(u)$ to make clear that it depends on $\omega _{1}$ and~$\omega _{2}$.

A ``shift'' of $\sigma$ by $a\omega_{1}+b\omega_{2}$ is defined by 
\begin{multline}\label{eq:sigma-ab}
\sigma_{a,b}(u\mid \omega_{1},\omega_{2})
\\
=\exp\Bigl((a\eta_{1}+b\eta_{2})\bigl(u-\tfrac{a\omega_{1}+b\omega_{2}}{2}\bigr)\Bigr)
\cdot \sigma(u - a\omega_{1} - b\omega_{2}\mid \omega_{1},\omega_{2}).
\end{multline}
Clearly, $\sigma_{a,b}(u)$ has simple zeros at $u\equiv a\omega_{1}+b\omega_{2}\mod \Lambda$.

Let $N$ be an integer, and consider the overlattice $\Lambda_{N}=\left\<\tfrac{\omega_{1}}{N},\omega_{2}\right\>\supset \Lambda$.
Hurwitz defines, for any integer $k$ 
\[
X_{k}(u)=\mu_{k}\cdot \exp({-G_{1}u^{2}})\cdot
\sigma_{\epsilon,\epsilon+\frac{k}{N}}
\bigl(u\,\big|\,\tfrac{\omega_{1}}{N},\omega_{2}\bigr),
\]
where $\epsilon =0$ if $N$ odd, and $\epsilon =\tfrac{1}{2}$ if $N$ even.
The quantity $G_{1}$ is given by the formula
$G_{1} = \frac{N\bar\eta_{1}-N\eta_{1}}{2\omega_{1}}
=\frac{\bar\eta_{2}-N\eta_{2}}{2\omega_{2}}$, 
where $\bar\eta_{1}$ and $\bar\eta_{2}$ denote the values that $\eta_{1}$ and $\eta_{2}$ take when $\omega_{1}$ is replaced by $\omega_{1}/{N}$.
Finally, Hurwitz defines  $\mu_{k}$ by $\mu_{k} =
e^{-\tfrac{\pi i}{2N}k}\cdot e^{-\tfrac{5\pi i}{4}}\cdot\mu$
 if $N$ even, and 
$(-1)^{k}\cdot \mu$ if $N$ odd, where $\mu$ is an arbitrary constant.

The function $X_{k}(u)$ has $N$ simple zeros within the fundamental domain $\{r\omega_{1}+s\omega_{2}\mid 0\le r,s<1\}$ of $\Lambda$. They are 
\begin{equation}\label{eq:sigma-zeros}
\begin{array}{cll}
\tfrac{m}{N}\omega_{1}+\tfrac{k}{N}\omega_{2}, & m=0,1,\dots,N-1,
&\ \text{ if $N$ is odd,} \\[\medskipamount]
\bigl(\tfrac{1}{2N} + \tfrac{m}{N}\bigr)\omega_{1}
+\bigl(\tfrac{1}{2} + \tfrac{k}{N}\bigr)\omega_{2}, & m=0,1,\dots,N-1,
&\ \text{ if $N$ is even.}
\end{array}
\end{equation}
Hurwitz then shows, among other things, the following relation for $\lambda_{1}$, $\lambda_{2}\in \Z$.
\begin{multline*}
X_{k}\Bigl(u+\tfrac{\lambda_{1}\omega_{1}+\lambda_{2}\omega{2}}{N}\Bigr)
=(-1)^{N(\lambda_{1}+\lambda_{2})}\cdot
\exp\Bigl(-2\pi i\Bigl(\tfrac{k\lambda_{1}}{N}+\tfrac{\lambda_{1}\lambda_{2}}{2N}\Bigr)\Bigr)
\\
\times
\exp\Bigl((\lambda_{1}\eta_{1}+\lambda_{2}\eta_{2})\Bigl(u+\tfrac{\lambda_{1}\omega_{1}+\lambda_{2}\omega_{2}}{2N}\Bigr)\Bigr)\cdot
X_{k-\lambda_{2}}(u)
\end{multline*}     
In particular, if $\lambda_{1}=0$ and $\lambda_{2}=k$, we have
\begin{equation}\label{eq:Hurwitz-0-k}
X_{k}(u)=(-1)^{Nk}\cdot
\exp\Bigl(k\eta_{2}\bigl(u-\tfrac{k\omega_{2}}{2N}\bigr)\Bigr)\cdot
X_{0}\bigl(u-\tfrac{k\omega_{2}}{N}\bigr).
\end{equation}

On the other hand, we can show using Proposition~\ref{prop:theta-basic2} (4) that our function $\theta_{k}(z)
=\theta_{(\frac{1}{2}-\frac{k}{N},\frac{N}{2})}(Nz,N\tau)$ satisfies
\begin{equation}\label{eq:Theta-0-k}
\theta_{k}(z) = (-1)^{k}\cdot
\exp\Bigl(-2\pi i k\bigl(z-\tfrac{k\tau}{2N}\bigr)\Bigr)\cdot
\theta_{0}\bigl(z-\tfrac{k\tau}{N}\bigr).
\end{equation}
Using \eqref{eq:Hurwitz-0-k} and \eqref{eq:Theta-0-k}, we can deduce the relation between $X_{k}(u)$ and $\theta_{k}(z)$ for any $k$ once we establish the relation between Hurwitz's $X_{0}(u)$ and $\theta_{0}(z)$.

Define $\varphi_{\epsilon,\epsilon}(u)$ by
\[
\varphi_{\epsilon,\epsilon}(u) = 
\exp\Bigl(\tfrac{-N\bar\eta_{1}}{2\omega_{1}}u^{2}\Bigr)\cdot
\sigma_{\epsilon,\epsilon}
\bigl(u\,\big|\,\tfrac{\omega_{1}}{N},\omega_{2}\bigr).
\]
Then, we have $X_{0}(u)=\mu_{0}\cdot 
\exp\Bigl(\tfrac{N\eta_{1}}{2\omega_{1}}u^{2}\Bigr)
\cdot \varphi_{\epsilon,\epsilon}(u)$.
We then compute $\varphi_{\epsilon,\epsilon}(u+\omega_{1}/N)$ and $\varphi_{\epsilon,\epsilon}(u+\omega_{2})$ using the formula \eqref{eq:sigma-ab},
\begin{align*}
&\varphi_{\epsilon,\epsilon}\bigl(u+\tfrac{\omega_{1}}{N}\bigr)
=-\exp(-2\pi i \epsilon)\cdot\varphi_{\epsilon,\epsilon}(u),
\\
&\varphi_{\epsilon,\epsilon}\bigl(u+\omega_{2}\bigr)
=-\exp\Bigl(-2\pi i \bigl(\tfrac{Nu}{\omega_{1}}
	+\tfrac{N\omega_{2}}{2\omega_{1}}\bigr)+2\pi i \epsilon\Bigr)
\cdot\varphi_{\epsilon,\epsilon}(u).
\end{align*}
Let $\tau=\omega_{2}/\omega_{1}$ and $u=\omega_{1}z$, and define 
\(
\Phi_{\epsilon_{1},\epsilon_{2}}(z)
=\varphi_{\epsilon,\epsilon}(\omega_{1}z).
\)
Then, we have
\begin{equation}\label{eq:Phi-ee}
\begin{aligned}
&\Phi_{\epsilon,\epsilon}\bigl(z+\tfrac{1}{N}\bigr)
=-\exp(-2\pi i \epsilon)
\cdot
\Phi_{\epsilon,\epsilon}(z),
\\
&\Phi_{\epsilon,\epsilon}(z+\tau)
=-\exp\Bigl(-2\pi i \bigl(Nz+\tfrac{N\tau}{2}\bigr)
	+2\pi i \epsilon\Bigr)
\cdot\Phi_{\epsilon,\epsilon}(z).
\end{aligned}
\end{equation}

\subsection{Odd case}
By Proposition~\ref{prop:theta-basic2} (2) and (3), $\theta_{0}(z)$ satisfies
\begin{align*}
&\theta_{0}\bigl(z+\tfrac{1}{N}\bigr)
=-\theta_{0}(z),
\\
&\theta_{0}(z+\tau)
=
-\exp\Bigl(-2\pi i \bigl(Nz+\tfrac{N\tau}{2}\bigr)\Bigr)\cdot
\theta_{0}(z),
\end{align*}
which is exactly the same transformation rules \eqref{eq:Phi-ee} for $\epsilon=0$.  Thus,  if $N$ is odd, the ratio $\Phi_{0,0}(z)/\theta_{0}(z)$ is a doubly periodic function of periods $1/N$ and $\tau$.
Since the function $\Phi_{0,0}(z)/\theta_{0}(z)$ has at most one simple pole in the fundamental domain, it must be a constant~$C$.  The constant $C$ is obtained by L'H\^opital's rule as follows
\[
C=\lim_{z\to0}\frac{\Phi_{0,0}(z)}{\theta_{0}(z)}
=\frac{\omega_{1}\sigma_{0,0}'(0)}{N\theta_{(\frac{1}{2},\frac{N}{2})}'(0,N\tau)}
=\frac{\omega_{1}}{N\theta_{(\frac{1}{2},\frac{N}{2})}'(0,N\tau)}.
\]
Thus, if $N$ is odd, then
\(
X_{0}(\omega_{1}z)=
C\cdot\mu_{0}\cdot 
\exp\Bigl(\tfrac{N\eta_{1}\omega_{1}}{2}z^{2}\Bigr)
\cdot \theta_{0}(z).
\)
The ratio of $X_{k}(u)$ and $\theta_{k}(z)$ can be calculated using \eqref{eq:Hurwitz-0-k} and Proposition~\ref{prop:theta-basic2} (4),
\[
\frac{X_{k}(\omega_{1}z)}{\theta_{k}(z)}
=\frac{\exp\Bigl(k\eta_{2}\bigl(\omega_{1}z-\tfrac{k\omega_{2}}{2N}\bigr)\Bigr)}
{\exp\Bigl(-2\pi i k\bigl(z-\tfrac{k\tau}{2N}\bigr)\Bigr)}\cdot
\frac{X_{0}\bigl(\omega_{1}z-\tfrac{k\omega_{2}}{N}\bigr)}
{\theta_{0}\bigl(z-\tfrac{k\tau}{N}\bigr)}
=C\cdot\mu_{0}\cdot 
\exp\Bigl(\tfrac{N\eta_{1}\omega_{1}}{2}z^{2}\Bigr).
\]
This shows that two immersions by $X_{k}(u)$ and $\theta_{k}(z)$ coincide with each other.

\subsection{Even case}
If $N$ is even, the zeros of Hurwitz's $X_{0}(u)$ and the zeros of $\theta_{0}(u/\omega_{1})$ are off by $\frac{1}{2}\omega_{2}$, and so the two immersions differ.
We thus try to modify the definition of $\theta_{k}^{(N)}$ to shift the zeros.  It turns out that if we define an immersion using the function
\[
\theta^{*}_{k}(z)=(-1)^{k}\theta_{(\frac{k}{N},\frac{1}{2})}(Nz,N\tau) 
\]
instead of $\theta_{k}(z)=\theta_{(\frac{1}{2}-\frac{k}{N},\frac{N}{2})}(Nz,N\tau)$, the immersion coincides with that of Hurwitz's.
To see this, compare $X_{0}(u)$ and $\theta_{0}^{*}(z)$. In this case $\Phi_{\frac{1}{2},\frac{1}{2}}(z)=\varphi_{\frac{1}{2},\frac{1}{2}}(\omega_{1}z)$ satisfy
\begin{align*}
&\Phi_{\frac{1}{2},\frac{1}{2}}\bigl(z+\tfrac{1}{N}\bigr)
=-\exp(\pi i)\cdot\Phi_{\frac{1}{2},\frac{1}{2}}(z)
=\Phi_{\frac{1}{2},\frac{1}{2}}(z),
\\
&\Phi_{\frac{1}{2},\frac{1}{2}}(z+\tau)
=\exp\Bigl(-2\pi i \bigl(Nz+\tfrac{N\tau}{2}\bigr)\Bigr)
\cdot\Phi_{\frac{1}{2},\frac{1}{2}}(z).
\end{align*}
On the other hand, by Proposition~\ref{prop:theta-basic}, we can show 
\begin{align*}
&\theta_{0}^{*}\bigl(z+\tfrac{1}{N}\bigr)
=\theta_{(0,\frac{N}{2})}(Nz+1,N\tau)
=\theta_{0}^{*}(z),
\\
&\begin{aligned}
\theta_{0}^{*}(z+\tau)
&=\theta_{(0,\frac{N}{2})}\bigl(Nz+N\tau,N\tau\bigr)
=\exp\Bigl(-2\pi i \bigl(Nz+\tfrac{N\tau}{2}\bigr)\Bigr)\cdot
\theta_{0}^{*}(z).
\end{aligned}
\end{align*}
These identities imply that if $N$ is even, the ratio $\Phi(z)/\theta_{0}^{*}(z)$ is a doubly periodic function of periods $1/N$ and $\tau$. Since this function has at most one simple pole in the fundamental domaine, it must be a constant~$C^{*}$. Thus, if $N$ is even, then
$X_{0}(\omega_{1}z)
=
C^{*}\cdot\mu_{0}\cdot 
\exp\bigl(\tfrac{N\eta_{1}\omega_{1}}{2}z^{2}\bigr)
\cdot \theta_{0}^{*}(z)$.

By Proposition~6.2 (2) with values $p=0$, $q=N/2$, $r=-k/N$, we can show
\(
\theta_{k}^{*}(z) =
\exp\Bigl(-2\pi i k\bigl(z-\tfrac{k\tau}{2N}\bigr)\Bigr)\cdot
\theta_{0}^{*}\bigl(z-\tfrac{k\tau}{N}\bigr)
\).
Therefore, the ratio of the two functions becomes
\[
\frac{X_{k}(\omega_{1}z)}{\theta_{k}^{*}(z)}
=\frac{\exp\Bigl(k\eta_{2}\bigl(\omega_{1}z-\tfrac{k\omega_{2}}{2N}\bigr)\Bigr)}
{\exp\Bigl(-2\pi i k\bigl(z-\tfrac{k\tau}{2N}\bigr)\Bigr)}\cdot
\frac{X_{0}\bigl(\omega_{1}z-\tfrac{k\omega_{2}}{N}\bigr)}
{\theta_{0}^{*}\bigl(z-\tfrac{k\tau}{N}\bigr)}
=C^{*}\cdot\mu_{0}\cdot 
\exp\Bigl(\tfrac{N\eta_{1}\omega_{1}}{2}z^{2}\Bigr).
\]
This implies that two immersions by $X_{k}(u)$ and $\theta_{k}^{*}(z)$ coincide with each other.

Alternatively, we may modify Hurwitz's definition to make it coincides with ours.  To do so, we use $\sigma_{\frac{1}{2},0}
\bigl(u\,\big|\,\tfrac{\omega_{1}}{N},\omega_{2}\bigr)$, which has the same set of simple zeros with $\theta_{0}^{(N)}(z,\tau)$.

\def\arXiv#1{arXiv:\href{http://arXiv.org/abs/#1}{#1}}
\providecommand{\bysame}{\leavevmode\hbox to3em{\hrulefill}\thinspace}
\providecommand{\MR}{\relax\ifhmode\unskip\space\fi MR }
\providecommand{\MRhref}[2]{%
  \href{http://www.ams.org/mathscinet-getitem?mr=#1}{#2}
}
\providecommand{\href}[2]{#2}

\end{document}